\documentclass[draft]{article} 
\usepackage[english]{babel}

\usepackage[letterpaper,top=2cm,bottom=2cm,left=1.5cm,right=1.5cm,marginparwidth=1.75cm]{geometry}
\usepackage{calrsfs}
\usepackage{amssymb}
\usepackage{amsmath}
\usepackage{amsthm}
\usepackage{caption}
\usepackage{subcaption}
\usepackage{multirow}
\usepackage[final]{graphicx}
\usepackage{amsfonts}
\usepackage{mathtools}
\usepackage{mathrsfs}
\usepackage{placeins}
\usepackage{algorithm}
\usepackage{algpseudocode}
\usepackage{textgreek}

\graphicspath{{Images/}}

\newtheorem{remark}{Remark}

\newtheorem{assumption}{Assumption}
\newtheorem{theorem}{Theorem}
\newtheorem{proposition}{Proposition}

\usepackage[mathscr]{euscript}
\usepackage[dvipsnames]{xcolor}
\usepackage[colorlinks=true, allcolors=OliveGreen]{hyperref}
\usepackage{filecontents}
\usepackage{tikz}
\usepackage{pgfplots}
\usetikzlibrary{external}
\usepackage{authblk}
\usepackage{todonotes}
\usepackage{array}
\newcolumntype{C}{>{\centering\arraybackslash}p{2cm}}
\mathtoolsset{showonlyrefs}

\usepackage{orcidlink}
\newcommand{\avg}[1]{\{\!\!\{#1\}\!\!\}_{\delta}}
\newcommand{\avgl}[1]{\{\!\!\{#1\}\!\!\}_{\delta_l}}
\newcommand{\jump}[1]{[\![#1]\!]}
\newcommand{\tjump}[1]{[\![\![#1]\!]\!]}

\newcommand{\partition}{\mathcal{T}_h}
\newcommand{\facesinternal}{\mathcal{F}^\mathrm{I}_h}
\newcommand{\faces}{\mathcal{F}_h}
\newcommand{\facesN}{\mathcal{F}_h^\mathrm{N}}

\newcommand{\facesD}{\mathcal{F}_h^\mathrm{D}}

\newcommand{\facesboundary}{\mathcal{F}^\mathrm{B}_h}
\newcommand{\Wh}{\mathbf{W}_h^\mathrm{DG}}
\DeclareMathAlphabet{\mathcalligra}{T1}{calligra}{m}{n}
\newcommand{\dtau}{\mathrm{d}\tau}
\newcommand{\Lpnorm}[2]{\|{#1}\|_{\mathbf{L}^{#2}(\Omega)}}
\newcommand{\Lpnormsc}[2]{\|{#1}\|_{L^{#2}(\Omega)}}
\newcommand{\DGnorm}[1]{\|{#1}\|_{\mathrm{DG}}}
\newcommand{\TDGnorm}[1]{|\!\!\;|\!\!\;|{#1}|\!\!\;|\!\!\;|_{\mathrm{DG}}}
\newcommand{\eonenorm}[1]{\|{#1}\|_{e_1}}
\newcommand{\etwonorm}[1]{\|{#1}\|_{e_2}}
\newcommand{\CGN}[1]{C_{\mathrm{GN}_{#1}}}

\newcommand{\Cen}[2]{C_{e_{#1}}^{({#2})}}
\renewcommand{\mathcalligra}[1]{\mathcal{#1}}

\allowdisplaybreaks[3]

\title{A weighted polygonal discontinuous Galerkin method for hierarchically coupled reaction–diffusion systems with cubic interactions\footnote{\textbf{Funding}: This work is partially funded by the European Union (ERC SyG, NEMESIS, project number 101115663). Views and opinions expressed are, however, those of the authors only and do not necessarily reflect those of the European Union or the European Research Council Executive Agency. Neither the European Union nor the granting authority can be held responsible for them. The author acknowledges “INdAM - GNCS Project”, code CUP E53C25002010001. The present research is part of the activities of the Dipartimento di Eccellenza 2023-2027 grant, funded by MUR. The author is member of INdAM-GNCS. }}

\author[1]{Mattia Corti \orcidlink{0000-0002-7014-972X}}

\affil[1]{MOX-Dipartimento di Matematica, Politecnico di Milano, Piazza Leonardo da Vinci 32, Milan, 20133, Italy}

\begin{document}
\maketitle

\begin{abstract}
We develop and analyse a symmetric weighted interior penalty polytopal discontinuous Galerkin (SWIP-PolyDG) method for multi-species reaction--diffusion systems with nonlinear reaction terms of up to cubic order. Such terms arise naturally when higher-order interactions are incorporated into population models, allowing non-additive effects among multiple species to influence local growth and conversion mechanisms. The proposed framework accommodates heterogeneous and possibly anisotropic diffusion tensors on general polygonal meshes. To control the nonlinear coupling, we consider a hierarchical block structure in the reaction operator, whereby each population block depends only on its own variables and on those associated with preceding blocks. In addition, the cubic self-interactions within each block are assumed to have a dissipative diagonal structure. Under these hypotheses and proceeding recursively over the hierarchy of population blocks, we derive local-in-time stability estimates in two spatial dimensions for the semi-discrete formulation in both the $L^2(\Omega)$-and dG-norms. We further establish an a priori error estimate in a combined $L^2(\Omega)$--dG energy norm for sufficiently regular solutions for the semi-discrete formulation. Finally, numerical experiments confirm the predicted convergence behaviour and illustrate the robustness of the method under heterogeneous diffusion.
\end{abstract}
\section{Introduction}
Reaction--diffusion systems provide a flexible mathematical framework for describing the evolution of interacting populations, in which reaction terms encode local growth, competition, cooperation, and conversion mechanisms, while diffusion accounts for spatial dispersal. Such systems arise in a broad range of applications, including ecology \cite{he_effects_2013}, epidemiology \cite{lotfi_partial_2014}, chemical kinetics \cite{smoller_shock_1994}, and protein--protein interaction models \cite{fornari_spatially-extended_2020}.

A common modelling assumption is that population dynamics are governed by pairwise interactions, so that the per-capita growth rate of a population depends additively on the densities of the other populations, as in classical Lotka--Volterra \cite{he_effects_2013}. However, this assumption may be inadequate when the effect of an interaction between two populations is modified by the presence or density of additional populations \cite{terry_impact_2025,wootton_nature_1994}. Such non-additive effects are commonly referred to as higher-order interactions (HOIs) \cite{grilli_higher-order_2017}. 

In reaction--diffusion population models, HOIs can be represented by nonlinear terms in the per-capita growth rates. In particular, quadratic non-additive effects at the per-capita level give rise to cubic reaction terms in the equations governing population densities. Such terms may alter the coexistence, stability, and persistence properties of ecological communities; see, e.g.,~\cite{letten_mechanistic_2019,sing_higher_2021}. More generally, nonlinear interaction terms of this form also arise in epidemiological and biochemical models. In these settings, complex nonlinear interactions may be approximated through Taylor expansions of underlying reaction mechanisms; higher-order terms become relevant whenever a second-order truncation does not adequately capture the observed dynamics \cite{fornari_prion-like_2019}.

The inclusion of cubic reaction terms substantially increases the analytical complexity of multi-species reaction--diffusion systems. For locally Lipschitz reaction operators, standard semilinear parabolic theory yields local-in-time classical solutions. For superlinear reaction--diffusion systems, finite-time blow-up cannot be excluded in general, even under natural mass-dissipation assumptions. Conversely, global existence can be recovered under additional structural hypotheses, such as quasi-positivity, suitable mass-control or dissipation conditions, and growth restrictions on the reaction terms; see, e.g.,~\cite{pierre_global_2010,pierre_schmitt_blowup_2000,souplet_global_2018}.

Cubic reaction terms are particularly delicate because their treatment requires the control of higher Lebesgue norms and the careful use of nonlinear interpolation estimates. In the two-dimensional setting considered here, the relevant bounds follow from discrete Sobolev embeddings, which provide control of the $L^p$ norms for every finite $p$ in terms of the $H^1$ norm. This property is essential in estimating the products generated by the cubic reaction terms. By contrast, in three spatial dimensions, the corresponding Sobolev embedding is limited to $p \leq 6$, and does not in general provide the higher integrability required by the present argument. Consequently, the extension of our analysis to three dimensions remains an open issue.

Moreover, in multi-species systems, fully coupled cubic interactions may transfer growth between components and prevent a direct componentwise argument. In the present work, we address the corresponding difficulty at the semi-discrete level by assuming a hierarchical coupling structure, in which each population block depends only on its own variables and on the variables associated with preceding blocks. Together with a dissipative diagonal structure for the intra-block cubic interactions, this assumption allows the nonlinear reaction terms to be controlled recursively along the hierarchy.

On the numerical side, reaction--diffusion systems may develop localized concentrations, propagating fronts, sharp transition layers, and complex spatial patterns, thus requiring discretizations that combine high-order
accuracy with robustness under heterogeneous and possibly anisotropic diffusion
\cite{antonietti_optimized_2026,corti_discontinuous_2023}.
Discontinuous Galerkin (dG) methods are particularly attractive in this setting because they naturally support high-order polynomial approximations, accommodate nonconforming and polygonal or polyhedral meshes, and handle discontinuous material parameters \cite{cangiani_hp_version_2017}.
In particular, symmetric weighted interior penalty (SWIP) formulations offer additional advantages for diffusion operators with large coefficient jumps or anisotropy \cite{antonietti_structure-preserving_2026}: diffusivity-dependent weighted averages and penalty parameters scaled by harmonic averages of the normal diffusivities yield stable discretizations that are robust with respect to strong local diffusion contrasts \cite{ern_discontinuous_2009,bonetti_robust_2025}.

In this work, we propose and analyze a SWIP-PolyDG method for hierarchically coupled nonlinear reaction--diffusion systems with HOIs. The model comprises several interacting populations and heterogeneous diffusion tensors on general polygonal meshes. The nonlinear reaction operator is represented through second-, third-, and fourth-order interaction tensors. We assume a block-triangular hierarchical structure and a dissipative diagonal form for the cubic self-interactions within each block. These hypotheses allow the nonlinear reaction terms to be controlled recursively along the hierarchy. The main contributions of this work are as follows:
\begin{itemize}
    \item We formulate a symmetric weighted interior penalty polytopal discontinuous Galerkin (SWIP-PolyDG) discretization for multi-species reaction--diffusion systems with reaction terms of up to cubic order, posed on general polygonal meshes and featuring heterogeneous diffusion tensors. We prove the coercivity and continuity of the discrete diffusion bilinear form for sufficiently large penalty parameters based on harmonic averages.

    \item We establish local-in-time stability estimates both in $L^2$- and dG-norms for the semi-discrete formulation through a recursive argument over the hierarchy of population blocks. As a byproduct, the first-block analysis yields stability estimates for general multi-species reaction--diffusion systems with linear and quadratic interactions, encompassing several classical models in mathematical biology and physical chemistry.

    \item We derive an a priori error estimate in a combined $L^2(\Omega)$--dG energy norm for sufficiently regular solutions.
    In the case of purely quadratic nonlinearities, the result extends
    the analysis of the Fisher--Kolmogorov equation in \cite[Theorem 2]{corti_discontinuous_2023} to vector-valued systems, while removing the structural relation previously required
    between the diffusion coercivity constant and the reaction terms.

    \item We validate the proposed method on a three-species Lotka--Volterra competition--diffusion system and on a manufactured     two-species problem featuring cubic HOIs. We further illustrate the     robustness of the method for a coupled Fisher--Kolmogorov system with discontinuous diffusion coefficients.
\end{itemize}
The remainder of the paper is organized as follows. In Section~\ref{sec:model_general}, we introduce the class of hierarchically coupled nonlinear reaction--diffusion systems and state the assumptions on the interaction tensors. Section~\ref{sec:polydg} presents the weighted interior penalty PolyDG formulation and its main coercivity and continuity properties. The local-in-time stability analysis and the a priori error estimate for the semi-discrete method are derived in Sections~\ref{sec:stability} and~\ref{sec:error_analysis}, respectively. Section~\ref{sec:IMEX_RK} describes the IMEX--Runge--Kutta time discretization used in the numerical experiments. Finally, the numerical results are presented in Section~\ref{sec:numerical_results}, followed by concluding remarks in Section~\ref{sec:conclusion}.

\section{Mathematical model}
\label{sec:model_general}
In this article, we consider a hierarchically coupled multi-species nonlinear reaction-diffusion system, modeling ecological and epidemiological dynamics with higher‑order interaction terms up to cubic order. For this reason we consider a variable $\boldsymbol{w} = \boldsymbol{w}(t,\boldsymbol{x})$  to model the population dynamics.
The general system of equations reads:
\begin{equation}
    \label{eq:general_strong}
    \begin{dcases}
    \dfrac{\partial \boldsymbol{w}}{\partial t} = \nabla \cdot (\mathbb{D} : \nabla \boldsymbol{w}) + \mathbf{G}(\boldsymbol{w})\boldsymbol{w} + \boldsymbol{\gamma},
    & \mathrm{in}\:\Omega\times(0,T];
    \\
    (\mathbb{D} \nabla \boldsymbol{w})\,\boldsymbol{n}_\Omega=\boldsymbol{0},  & \mathrm{on}\: \Gamma_\mathrm{N}  \times(0,T];    
    \\
    \boldsymbol{w}=\boldsymbol{0},  & \mathrm{on}\: \Gamma_\mathrm{D}  \times(0,T];
    \\
    \boldsymbol{w}(0,\boldsymbol{x})=\boldsymbol{w}_{0},  & \mathrm{in}\: \Omega. 
    \end{dcases}
\end{equation}
where $\Gamma_\mathrm{D}\cup\Gamma_\mathrm{N} = \partial \Omega$ with $\Gamma_\mathrm{D}\cap\Gamma_\mathrm{N} = \emptyset$ and $|\Gamma_\mathrm{D}|>0$. 
This system generalises the case of the dynamics of the family of $n$ populations $\boldsymbol{w}=[w_1,\dots,w_n]^\top$. In this context, $\mathbb{D}\in\mathbb{L}^\infty(\Omega;\mathbb{R}^{n\times d \times n \times d}$) is a fourth-order diffusion tensor such that: $\mathbb{D}_{kilj} = \mathbf{D}^l_{ij} \delta_{kl}$, with $\mathbf{D}^l\in\mathbb{L}^\infty(\Omega)$ symmetric diffusion tensor associated with the $l$-th population and $\delta$ is the Kronecker's delta. Moreover, we assume that $\exists\,d_0^l > 0$ such that $d_0^l |\boldsymbol{\psi}|^2 \leq \boldsymbol{\psi}^T \mathbf{D}^l \boldsymbol{\psi} \quad \forall \boldsymbol{\psi} \in \mathbb{R}^d$ and for a.e. $x\in\Omega$. Moreover, $\boldsymbol{\gamma}\in\mathbf{L}^2(\Omega)$ is an external forcing term and $\mathbf{G}:\mathbb{R}^n\rightarrow\mathbb{R}^n$ is a mapping associated with linear, quadratic and cubic components, as follows:
\begin{equation}
    \mathbf{G}(\boldsymbol{w})\boldsymbol{w} = (\mathbf{L} + \mathbf{Q} \boldsymbol{w} + (\mathbb{C} \boldsymbol{w}) \boldsymbol{w} )\boldsymbol{w}.
\end{equation}
In particular, we have that:
\begin{itemize}
    \item $\mathbf{L}\in\mathbb{R}^{n\times n}$ is a second order tensor (matrix) containing the coefficients of linear reaction terms $L_{ij}$;
    \item $\mathbf{Q}\in\mathbb{R}^{n\times n \times n}$ is a third order tensor containing the coefficients of quadratic reaction terms $Q_{ijk}$;
    \item $\mathbb{C}\in\mathbb{R}^{n\times n \times n \times n}$ is a fourth order tensor containing the coefficients of cubic reaction terms $C_{ijkl}$.
\end{itemize}
Then, we remark that the $k$-th equation of the system reads:
\[
\partial_t w_k = \nabla \cdot (\mathbf{D}^k \nabla w_k) + \gamma_k + \sum_{i=1}^{n} G_{ki}  w_i + \sum_{i,j=1}^{n} Q_{kij} w_i w_j + \sum_{i,j,l=1}^{n} C_{kijl} w_i w_j w_l.
\]
\par
\subsection{The block structure}
In this work, we assume that the fourth-order tensor $\mathbb{C}$ induces on the system a block structure dividing the vectors in $p_\mathrm{M}$ blocks:
\begin{equation}
    \boldsymbol{w} = [\boldsymbol{w}^{(1)},\dots,\boldsymbol{w}^{(p)},\dots,\boldsymbol{w}^{(p_\mathrm{M})}]^T, \qquad \boldsymbol{\gamma} = [\boldsymbol{\gamma}^{(1)},\dots,\boldsymbol{\gamma}^{(p)},\dots,\boldsymbol{\gamma}^{(p_\mathrm{M})}]^T, \qquad \text{with } \boldsymbol{w}^{(p)},\boldsymbol{\gamma}^{(p)} \in \mathbb{R}^{n_p},
\end{equation}
with $\sum_{p=1}^{n_{p_\mathrm{M}}} n_p = n$. Moreover, we denote by $\mathbf{L}^{(p,q)} \in \mathbb{R}^{n_p\times n_q}$, $\mathbf{Q}^{(p,q,s)} \in \mathbb{R}^{n_p\times n_q \times n_s}$, and $\mathbb{C}^{(p,q,s,t)} \in \mathbb{R}^{n_p\times n_q \times n_s \times n_t}$. Finally, the system of equations associated with the $p$-th block reads
\begin{equation}
\label{eq:general_block_strong}
\dfrac{\partial \boldsymbol{w}^{(p)}}{\partial t} = \nabla \cdot (\mathbb{D}^p : \nabla \boldsymbol{w}^{(p)}) + \sum_{q=1}^{p} \left(\mathbf{L}^{(p,q)} + \sum_{s=1}^{p} \left(\mathbf{Q}^{(p,q,s)} + \sum_{r=1}^{p} \mathbb{C}^{(p,q,s,r)}\boldsymbol{w}^{(r)} \right)\boldsymbol{w}^{(s)}\right) \boldsymbol{w}^{(q)} + \boldsymbol{\gamma}^{(p)}.
\end{equation}
From Equation \eqref{eq:general_block_strong}, it is clear that the reaction terms depend only on the quantities of the same or previous blocks. In particular, we assume an additional property:
\begin{assumption}[Diagonal blocks structure]
\label{ass:c_structure}
For any $p = 1,\dots,p_\mathrm{M}$, the block $\mathbb{C}^{(p,p,p,p)}$ is diagonal, namely 
\begin{equation}
((\mathbb{C}^{(p,p,p,p)} \boldsymbol{w}^{(p)}) \boldsymbol{w}^{(p)} )\boldsymbol{w}^{(p)} = -\left[c^{(p)}_1 \left(w^{(p)}_1\right)^3,\,\dots,\,c^{(p)}_{n_p} \left(w^{(p)}_{n_p}\right)^3\right]^T;
\end{equation}
moreover, we have that $c^p_j\geq 0$ for each $j=1,\dots,n_p$.
\end{assumption} 
\subsection{Weak formulation}
First of all, let us define the Sobolev space $\mathbf{W} = \mathbf{H}^1_{\Gamma_\mathrm{D}}(\Omega)$. Then, we can write the weak formulation of the problem in Equation \eqref{eq:general_strong}. By multiplying the first equation of \eqref{eq:general_strong}, by a test function $\boldsymbol{v}$ and integrating by parts, we obtain:
Find $\boldsymbol{w}\in L^\infty((0,T),\mathbf{W})\cap H^1((0,T),\mathbf{W}')$ such that
\begin{equation}
    \label{eq:general_weak}
    \begin{dcases}
    (\dot{\boldsymbol{w}},\boldsymbol{v}) + (\mathbb{D} : \nabla \boldsymbol{w}, \nabla \boldsymbol{v}) = (\mathbf{G}(\boldsymbol{w})\boldsymbol{w},\boldsymbol{v}) + (\boldsymbol{\gamma},\boldsymbol{v}),
    & \forall \boldsymbol{v}\in \mathbf{W};
    \\
    \boldsymbol{w}(0,\boldsymbol{x})=\boldsymbol{w}_{0},  & \mathrm{in}\: \Omega. 
    \end{dcases}
\end{equation}
\begin{remark}
We underline that, due to the structure of tensor $\mathbb{D}$, the following relation holds:
 \begin{equation}
(\mathbb{D} : \nabla \boldsymbol{w}, \nabla \boldsymbol{v}) = \sum_{l=1}^{n} (\mathbf{D}^l  \nabla w_l, \nabla v_l)
\end{equation}
\end{remark}
\section{Polytopal discontinuous Galerkin semi-discrete formulation}
\label{sec:polydg}
In this section, we derive the interior penalty discontinuous Galerkin formulation of problem \eqref{eq:general_weak}. First of all, we introduce some preliminary estimates.
\subsection{Discrete setting and preliminary estimates}
Let us introduce a polytopic mesh partition $\partition$ of the domain $\Omega$ made of disjoint polygonal/polyhedral (polytopal) elements $K$, where for each element $K\in \partition$, we denote by $|K|$ the measure of the element and by $h_K$ its diameter. We set $h=\max_{K\in\partition} h_K<1$. We define the interface as the intersection of the $(d-1)-$dimensional facets of two neighbouring elements. We distinguish two cases:
\begin{itemize}
    \item case $d=2$, in which the interfaces are always line segments; then we denote such a set of segments with $\faces$.
    \item case $d=3$, in which any interface consists of a generic polygon, we further assume that we can decompose each interface into (planar) triangles; we denote the set of all these triangles with $\faces$;
\end{itemize}
It is now useful to decompose $\faces$ into the union of interior faces ($\facesinternal$) and exterior faces ($\facesboundary$ ) lying on the boundary of the domain $\partial\Omega$, i.e., $\faces = \facesinternal \cup \facesboundary$. Moreover, we decompose the boundary faces into the set of faces associated with Dirichlet boundary conditions ($\facesD$) and with Neumann boundary conditions ($\facesN$).
\par
\begin{assumption}
\label{ass:mesh}
The mesh sequence $\{\partition\}_h$ satisfies the following properties \cite{di_pietro_hybrid_2020}:
\begin{enumerate}
    \item Shape Regularity: $\forall K\in\partition\;it\;holds: c_1 h_K^d\lesssim q|K|\lesssim  c_2h_K^d$.
    \item Contact Regularity: $\forall F\in\faces$ with $F\subseteq \overline{K}$ for some $K\in\partition$, it holds $h_K^{d-1}\lesssim |F|$, where $|F|$ is the Hausdorff measure of the face $F$.
    \item Submesh Condition: There exists a shape-regular, conforming, matching simplicial submesh $\widetilde{\partition}$ such that:
    \begin{itemize}
        \item $\forall \widetilde{K}\in\widetilde{\partition}\;\exists K\in\partition:\quad \widetilde{K}\subseteq K$.
        \item The family $\{\widetilde{\partition}\}_h$ is shape and contact regular.
        \item $\forall \widetilde{K}\in\widetilde{\partition}, K\in\partition$ with $\widetilde{K} \subseteq K$, it holds $h_K \lesssim h_{\widetilde{K}}$.
    \end{itemize}
\end{enumerate}
\end{assumption}
Let us define $\mathbb{P}_{\ell}(K)$ as the space of polynomials of total degree $\ell\geq 1$ over a mesh element $K$. Then, we can introduce the following discontinuous finite element space:
\begin{equation*}
    \Wh = \{\boldsymbol{w}\in \mathbf{L}^2(\Omega):\quad \boldsymbol{w}|_K\in[\mathbb{P}_{\ell}(K)]^n\quad\forall K\in\partition\}
\end{equation*}
\par
We next introduce the so-called trace operators \cite{arnoldUnifiedAnalysisDiscontinuous2001}. Let $F\in\facesinternal$ be a face shared by the elements $K^\pm$. Let $\boldsymbol{n}^\pm$ be the unit normal vector on face $F$ pointing exterior to $K^\pm$. Then, for sufficiently regular scalar-valued functions $v$, vector-valued functions $\boldsymbol{q}$, and tensor-valued functions $\boldsymbol{\tau}$, respectively, we define:
\begin{itemize}
   	\item the weighted average operator $\avg{\cdot}$ on $F\in \facesinternal$: 
\[\avg{v} = \delta^+ v^+ + \delta^- v^-, \quad \avg{\boldsymbol{q}} = \delta^+\boldsymbol{q}^+ + \delta^-\boldsymbol{q}^- \quad \avg{\boldsymbol{\tau}} = \delta^+\boldsymbol{\tau}^+ + \delta^-\boldsymbol{\tau}^-;
\]
    \item the jump operator $\jump{\cdot}$ on $F\in \facesinternal$: $\jump{v} = v^+\boldsymbol{n}^+ + v^-\boldsymbol{n}^-, \quad \jump{\boldsymbol{q}} = \boldsymbol{q}^+\cdot\boldsymbol{n}^+ + \boldsymbol{q}^-\cdot\boldsymbol{n}^-$;
    \item the tensor-jump operator $\tjump{\cdot}$ on $F\in \facesinternal$: $\tjump{\boldsymbol{q}} = \boldsymbol{q}^+\otimes\boldsymbol{n}^+ + \boldsymbol{q}^-\otimes\boldsymbol{n}^-$.
\end{itemize}
In these relations, we are using the superscripts $\pm$ to denote the traces of the functions on $F$ taken within the interior of the elements $K^\pm$, respectively. In the following sections, the weighted average parameters are chosen according to the physical diffusion parameter
\begin{equation}
\delta_l^\pm = \dfrac{\boldsymbol{n}^\mp\mathbf{D}^l|_{K^\mp}\boldsymbol{n}^\mp}{\boldsymbol{n}^+\mathbf{D}^l|_{K^+}\boldsymbol{n}^+ +\boldsymbol{n}^-\mathbf{D}^l|_{K^-}\boldsymbol{n}^-}.
\end{equation} 
Moreover, the diffusion tensors are assumed to be elementwise constant with respect to the computational mesh: $\mathbf{D}^l|_K \in\mathbb{R}^{d\times d}_{\mathrm{sym}}$ $\forall K\in\mathcal{T}_h$ and jumps in the diffusion coefficients are allowed across mesh interfaces. This choice gives rise to a robust formulation with respect to anisotropic and locally small structure of the diffusion tensor (see \cite{ern_discontinuous_2009}).
Moreover, let $F\in\facesD$ be a face associated with an element $K$ and with outward unit normal vector $\boldsymbol{n}$. Then, we define:
\begin{itemize}
   	\item the weighted average operator $\avg{\cdot}$ on $F\in \facesD$: $\avg{v} = v, 
    \quad \avg{\boldsymbol{q}} = \boldsymbol{q},
    \quad \avg{\boldsymbol{\tau}} = \boldsymbol{\tau};
$
    \item the jump operator $\jump{\cdot}$ on $F\in \facesD$: $\jump{v} = v\boldsymbol{n}, \quad \jump{\boldsymbol{q}} = \boldsymbol{q}\cdot\boldsymbol{n}$;
    \item the tensor-jump operator $\tjump{\cdot}$ on $F\in \facesD$: $\tjump{\boldsymbol{q}} = \boldsymbol{q}\otimes\boldsymbol{n}$.
\end{itemize}
\subsection{Semi-discrete formulation}
To construct the semi-discrete formulation, we define the following penalization face-wise function $\eta_l:\facesinternal\cup\facesD\rightarrow\mathbb{R}_+$ with $l=1,\dots,n$  defined as
\begin{equation}
\label{eq:penalty_def}
    \eta_l = \tilde{\eta}  \dfrac{|F|\,\ell^2\,\{d_l^K\}_\mathrm{H}}{\{|K|\}_\mathrm{H}},\; \mathrm{on}\,F\in\facesinternal, \quad \mathrm{and} \quad \eta_l = \tilde{\eta}  \dfrac{|F|\,\ell^2\,d_l^K}{|K|},\; \mathrm{on}\;F\in\facesD
\end{equation}
where we are considering the harmonic average operator $\{\cdot\}_\mathrm{H}$ on $K^\pm$ and $d_l^{K^\pm} = (\boldsymbol{n}^\pm)^\top \mathbf{D}^l|_{K^\pm}\boldsymbol{n}^\pm$ for any $K\in\partition$\footnote{In this context $\|\cdot\|_2$ is the operator norm induced by the $L^2$ norm in the space of symmetric second order tensors.} and $\tilde{\eta}$ is a parameter at our disposal (to be chosen large enough). Using the preliminary definitions in the previous section, we introduce the bilinear form $\mathcal{A}(\boldsymbol{w},\boldsymbol{v})$ such that:
\begin{equation}
\label{eq:bilinear_form_def}
\mathcal{A}(\boldsymbol{w},\boldsymbol{v}) = (\mathbb{D} \nabla_h \boldsymbol{w}, \nabla_h \boldsymbol{v}) - (\avg{\mathbb{D} \nabla_h \boldsymbol{w}}, \tjump{\boldsymbol{v}})_{\facesinternal\cup\facesD} -  (\tjump{\boldsymbol{w}}, \avg{\mathbb{D} \nabla_h \boldsymbol{v}})_{\facesinternal\cup\facesD} + (\mathbf{H} \tjump{\boldsymbol{w}}, \tjump{\boldsymbol{v}})_{\facesinternal\cup\facesD},
\end{equation}
where $\mathbf{H} = \mathrm{diag}(\eta_1,\dots,\eta_n)$. By exploiting the definition of the bilinear form, we obtain the following semi-discrete PolyDG formulation.
\par
\bigskip
Find $\boldsymbol{w}_h=\boldsymbol{w}_h(t)\in \Wh$ such that $\forall t>0$:
\begin{equation}
    \label{eq:semidiscrete}
    \begin{dcases}
    (\dot{\boldsymbol{w}}_h,\boldsymbol{v}_h) + \mathcal{A}(\boldsymbol{w}_h, \boldsymbol{v}_h) = (\mathbf{G}(\boldsymbol{w}_h)\boldsymbol{w}_h,\boldsymbol{v}_h) + (\boldsymbol{\gamma},\boldsymbol{v}_h),
    & \forall \boldsymbol{v}_h\in \Wh;
    \\
    \boldsymbol{w}_h(0)=\boldsymbol{w}_{0h}. 
    \end{dcases}
\end{equation}
\begin{proposition}[Coercivity]
\label{prop:coercivity}
The bilinear form $\mathcal{A}$, defined in Eq. \eqref{eq:bilinear_form_def}, is coercive, namely it satisfies
\begin{equation}
\mathcal{A}(\boldsymbol{v}_h, \boldsymbol{v}_h) \geq \frac{1}{2}\|\boldsymbol{v}_h\|^2_\mathrm{DG} \qquad \forall\boldsymbol{v}_h \in \Wh,
\end{equation}
under the assumption on the penalty parameter value $\tilde{\eta} \geq 4 C_\mathrm{inv}$, and $C_\mathrm{inv}$ is the inverse trace inequality constant (see \cite[Lemma 11]{cangiani_hp_version_2017}).
\end{proposition}
\begin{proof}
First of all, we need to notice that we can decompose the bilinear form in a sum of $n$ bilinear forms associated with the $l$-th scalar problem:
\begin{equation}
\label{eq:decomposition_A}
\begin{aligned}
\mathcal{A}(\boldsymbol{u},\boldsymbol{v}) = & (\mathbb{D} \nabla_h \boldsymbol{u}, \nabla_h \boldsymbol{v}) - (\avg{\mathbb{D} \nabla_h \boldsymbol{u}}, \tjump{\boldsymbol{v}})_{\facesinternal\cup\facesD} - (\tjump{\boldsymbol{u}}, \avg{\mathbb{D} \nabla_h \boldsymbol{v}})_{\facesinternal\cup\facesD} + (\mathbf{H} \tjump{\boldsymbol{u}}, \tjump{\boldsymbol{v}})_{\facesinternal\cup\facesD} \\ = & \sum_{l=1}^{n} \Big[(\mathbf{D}^l \nabla_h u_l, \nabla_h v_l) - (\avgl{\mathbf{D}^l \nabla_h u_l}, \jump{v_l})_{\facesinternal} - (\jump{u_l}, \avgl{\mathbf{D}^l \nabla_h v_l})_{\facesinternal} + (\eta_l \jump{u_l}, \jump{v_l})_{\facesinternal} \\ & \quad - (\mathbf{D}^l \nabla_h u_l, v_l)_{\facesD} - (u_l,\mathbf{D}^l \nabla_h v_l)_{\facesD} + (\eta_l u_l, v_l)_{\facesD} \Big].
\end{aligned}
\end{equation}
Then, by considering $\boldsymbol{u}=\boldsymbol{v}=\boldsymbol{v}_h$, we obtain the following problem:
\begin{equation}
\label{eq:coer:tested}
\mathcal{A}(\boldsymbol{v}_h,\boldsymbol{v}_h) = \sum_{l=1}^{n} \left[\|\sqrt{\mathbf{D}^l} \nabla_h v_{lh}\|^2_2 - 2(\avgl{\mathbf{D}^l \nabla_h v_{lh}}, \jump{v_{lh}})_{\facesinternal} - 2(\mathbf{D}^l \nabla_h v_{lh}, v_{lh})_{\facesD} + \|\eta_l^{1/2} \jump{v_{lh}}\|^2_{\facesinternal\cup\facesD} \right].
\end{equation}
Now, we have to treat the following quantity
\begin{equation}
\begin{aligned}
 |2(\avgl{\mathbf{D}^l \nabla_h v_{lh}}, \jump{v_{lh}})_{\facesinternal}| = & |2((\delta^+_l \mathbf{D}^l_+ \nabla_h v_{lh}^+,\jump{v_{lh}})_{\facesinternal}+ 2((\delta^-_l \mathbf{D}^l_- \nabla_h v_{lh}^-,\jump{v_{lh}})_{\facesinternal}| \\
=  & |2((\sqrt{\mathbf{D}^l}_+ \nabla_h v_{lh}^+, \delta^+_l \sqrt{\mathbf{D}^l}_+ \jump{v_{lh}})_{\facesinternal}+ 2((\sqrt{\mathbf{D}^l}_-\nabla_h v_{lh}^-, \delta^-_l \sqrt{\mathbf{D}^l}_- \jump{v_{lh}})_{\facesinternal}| \\
\text{(H\"{o}lder and Young inequalities)}\:\:\leq &  \sum_{F\in\facesinternal} \alpha_F^+ \|\sqrt{\mathbf{D}^l}_+ \nabla_h v_{lh}^+\|_{\mathbf{L}^2(F)}^2 + \alpha_F^-\|\sqrt{\mathbf{D}^l}_- \nabla_h v_{lh}^-\|_{\mathbf{L}^2(F)}^2 \\[-2pt] & \quad + (\alpha^-_F)^{-1}\|\delta^-_l \sqrt{\mathbf{D}^l}_-\jump{v_{lh}}\|_{\mathbf{L}^2(F)}^2 +  (\alpha^+_F)^{-1} \|\delta^+_l \sqrt{\mathbf{D}^l}_+\jump{v_{lh}}\|_{\mathbf{L}^2(F)}^2
\end{aligned}
\end{equation}
Then, we can bound the first two terms at the right-hand side by using the inverse-trace-inequality:
\begin{equation}
\begin{aligned}
|2(\avgl{\mathbf{D}^l \nabla_h v_{lh}}, \jump{v_{lh}})_{\facesinternal}| \leq & \ell^2 \sum_{F\in\facesinternal} C_\mathrm{inv}\left(\alpha_F^+\dfrac{|F|}{|K_+|}\|\sqrt{\mathbf{D}^l} \nabla_h v_{lh}\|_{\mathbf{L}^2(K_+)}^2+\alpha_F^-\dfrac{|F|}{|K_-|}\|\sqrt{\mathbf{D}^l} \nabla_h v_{lh}\|_{\mathbf{L}^2(K_-)}^2\right)  \\[-2pt] & \qquad +  (\alpha^-_F)^{-1}\|\delta^-_l \sqrt{\mathbf{D}^l}_-\jump{v_{lh}}\|_{\mathbf{L}^2(F)}^2 +  (\alpha^+_F)^{-1}\|\delta^+_l \sqrt{\mathbf{D}^l}_+\jump{v_{lh}}\|_{\mathbf{L}^2(F)}^2.
\end{aligned}
\end{equation}
We choose the parameter
\begin{equation}
\alpha_F^\pm = \dfrac{|K_\pm|}{2|F|C_\mathrm{inv}\ell^2},
\end{equation}
Then we can rewrite
\begin{equation}
\label{eq:coer:laststep}
\begin{aligned}
 \sum_{F\in\facesinternal} & 2 C_\mathrm{inv}\ell^2\dfrac{|F|}{|K_-|}\|\delta^-_l \sqrt{\mathbf{D}^l}_-\jump{v_{lh}}\|_{\mathbf{L}^2(F)}^2 + 2 C_\mathrm{inv}\ell^2\dfrac{|F|}{|K_+|}\|\delta^+_l \sqrt{\mathbf{D}^l}_+\jump{v_{lh}}\|_{\mathbf{L}^2(F)}^2 \\ 
\leq & \sum_{F\in\facesinternal} \int_F \left(2 C_\mathrm{inv}\ell^2\left(\dfrac{|F|}{|K_-|}(\delta^-_l)^2 d_l^{K_-}+\dfrac{|F|}{|K_+|}(\delta^+_l)^2 d_l^{K_+}\right)|\jump{v_{lh}}|^2\right) \\ 
\leq & \sum_{F\in\facesinternal} \int_F \left(2 C_\mathrm{inv}\ell^2 \left\{d_l^{K}\right\}_{\mathrm{H}}\dfrac{|F|}{\left\{|K|\right\}_{\mathrm{H}}}|\jump{v_{lh}}|^2\right) =  \sum_{F\in\facesinternal} \int_F \left(\hat{\eta}_l |\jump{v_{lh}}|^2\right) = \|\hat{\eta}^{1/2}_l\jump{v_{lh}}\|^2_{\facesinternal}.
\end{aligned}
\end{equation}
Finally, by using equations \eqref{eq:coer:tested} and \eqref{eq:coer:laststep}, and assuming that $\tilde{\eta}>4 C_\mathrm{inv}$ (similar arguments apply to the Dirichlet boundary faces), we find a coercivity estimate
\begin{equation}
\mathcal{A}(\boldsymbol{v}_h,\boldsymbol{v}_h) \geq \sum_{l=1}^{n} \frac{1}{2}\left[\|\sqrt{\mathbf{D}^l} \nabla_h v_{lh}\|^2_2 + \|\eta_l^{1/2} \jump{v_{lh}}\|^2_{\facesinternal\cup\facesD} \right] = \sum_{l=1}^{n} \frac{1}{2}\|v_{lh}\|^2_\mathrm{DG} = \frac{1}{2}\|\boldsymbol{v}_h\|^2_\mathrm{DG}.
\end{equation}
\end{proof}
\begin{proposition}[Continuity]
\label{prop:continuity}
The bilinear form $\mathcal{A}$, defined in Eq.~\eqref{eq:bilinear_form_def}, is continuous, namely it satisfies
\begin{equation}
\left|\mathcal{A}(\boldsymbol{u}_h,\boldsymbol{v}_h)\right|
\leq 4\|\boldsymbol{u}_h\|_\mathrm{DG}\|\boldsymbol{v}_h\|_\mathrm{DG}
\qquad \forall\,\boldsymbol{u}_h,\boldsymbol{v}_h\in\Wh.
\end{equation}
\end{proposition}
\begin{proof}
First of all, we start from Equation \eqref{eq:decomposition_A} by considering $\boldsymbol{u}=\boldsymbol{u}_h$ and $\boldsymbol{v}=\boldsymbol{v}_h$. Using the H\"older inequality we obtain:
\begin{equation}
(\sqrt{\mathbf{D}^l} \nabla_h u_{lh}, \sqrt{\mathbf{D}^l} \nabla_h v_{lh}) \leq \|\sqrt{\mathbf{D}^l} \nabla_h u_{lh}\|_2\|\sqrt{\mathbf{D}^l} \nabla_h v_{lh}\|_2,
\end{equation}
and
\begin{equation}
(\eta_l \jump{u_{lh}}, \jump{v_{lh}})_{\facesinternal} \leq \|\eta_l^{1/2} \jump{u_{lh}}\|_{\facesinternal}\|\eta_l^{1/2} \jump{v_{lh}}\|_{\facesinternal}.
\end{equation}
Now, we have to treat the following quantity 
\begin{equation}
\begin{aligned}
 |(\avgl{\mathbf{D}^l \nabla_h u_{lh}}, \jump{v_{lh}})_{\facesinternal}| = & |((\delta^+_l \mathbf{D}^l_+ \nabla_h u_{lh}^+,\jump{v_{lh}})_{\facesinternal}+ ((\delta^-_l \mathbf{D}^l_- \nabla_h u_{lh}^-,\jump{v_{lh}})_{\facesinternal}| \\
=  & |((\sqrt{\mathbf{D}^l}_+ \nabla_h v_{lh}^+, \delta^+_l \sqrt{\mathbf{D}^l}_+ \jump{v_{lh}})_{\facesinternal}+ ((\sqrt{\mathbf{D}^l}_-\nabla_h v_{lh}^-, \delta^-_l \sqrt{\mathbf{D}^l}_- \jump{v_{lh}})_{\facesinternal}| \\
\text{(H\"{o}lder inequality)}\:\:\leq &  \sum_{F\in\facesinternal} \left(\|(\alpha_F^+)^{1/2}\sqrt{\mathbf{D}^l}_+ \nabla_h u_{lh}^+\|_{\mathbf{L}^2(F)}^2 + \|(\alpha_F^-)^{1/2}\sqrt{\mathbf{D}^l}_- \nabla_h u_{lh}^-\|_{\mathbf{L}^2(F)}^2\right)^{1/2} \\[-2pt] & \quad \left(\|(\alpha^-_F)^{-1/2}\delta^-_l \sqrt{\mathbf{D}^l}_-\jump{v_{lh}}\|_{\mathbf{L}^2(F)}^2 +   \|(\alpha^+_F)^{-1/2}\delta^+_l \sqrt{\mathbf{D}^l}_+\jump{v_{lh}}\|_{\mathbf{L}^2(F)}^2\right)^{1/2}
\end{aligned}
\end{equation}
Finally, using the inverse-trace-inequality and the same steps of the coercivity proof we arrive at the following result
\begin{equation*}
    (\avgl{\mathbf{D}^l \nabla_h u_{lh}}, \jump{v_{lh}})_{\facesinternal}| \leq \|\sqrt{\mathbf{D}^l} \nabla_h u_{lh}\|_2\|\eta^{1/2}_l\jump{v_{lh}}\|_{\facesinternal}.
\end{equation*}
Repeating for the last term and for the Dirichlet boundary faces, we derive the thesis.
\end{proof}
Finally, we recall a specific type of Gagliardo-Nirenberg inequality that will be useful in the subsequent analysis.
\begin{proposition}[Gagliardo-Nirenberg-type $L^p(\Omega)$-interpolation]
The following inequality holds for $d=2$:
\begin{equation}
\label{eq:gn_adapt}
\Lpnorm{u}{p}^\alpha 
\leq \frac{\alpha(p-2)}{2p}\varepsilon_{p,\alpha}^{\frac{2p}{\alpha(p-2)}} \DGnorm{u}^{2} + \frac{2p-\alpha p +2\alpha}{2p}
\left(\frac{\CGN{p}^\alpha}{\varepsilon_{p,\alpha}}\right)^{\frac{2p}{2p-\alpha p +2\alpha}}
\Lpnorm{u}{2}^{\frac{4\alpha}{2p+2\alpha-\alpha p}},
\end{equation}
for any $\varepsilon_{p,\alpha}>0$ under the assumptions $p > 2$ and $\alpha \in (0, \frac{2p}{p-2})$.
\end{proposition}
\begin{proof}
    First of all, we recall that for $d=2$, the following Gagliardo-Nirenberg inequality holds \cite{gazca-orozco_stability_2025} for $d=2$:
    \begin{equation*}
        \Lpnorm{u}{p} \leq \CGN{p} \DGnorm{u}^\theta \Lpnorm{u}{2}^{1-\theta}, \quad \mathrm{with}\; \theta = 1-\frac{2}{p}.
    \end{equation*}
    Then, we can observe that for a general exponent $\alpha$ and for $p > 2$, we obtain:
    \begin{equation*}
        \Lpnorm{u}{p}^\alpha 
         \, \leq \left(\varepsilon_{p,\alpha} \DGnorm{u}^{\alpha\frac{p-2}{p}}\right) \, \left(\frac{\CGN{p}^\alpha}{\varepsilon_{p,\alpha}} \Lpnorm{u}{2}^{2\frac{\alpha}{p}}\right) \leq \frac{1}{q}\left(\varepsilon_{p,\alpha} \DGnorm{u}^{\alpha\frac{p-2}{p}}\right)^q + \frac{1}{q'}\left(\frac{\CGN{p}^\alpha}{\varepsilon_{p,\alpha}} \Lpnorm{u}{2}^{2\frac{\alpha}{p}}\right)^{q'},
    \end{equation*}
    where in the last step we used the Young's inequality and $\frac{1}{q}+\frac{1}{q'}=1$. Finally, we want that $\frac{\alpha(p-2) q}{p} = 2$, then $q = \frac{2p}{\alpha(p-2)}$ and $q' = \frac{2p}{2p-\alpha p +2\alpha}$. Substituting, we obtain the thesis.
\end{proof}
Finally, we report a small extension of the Perov inequality \cite{webb_extensions_2018}, that allows to handle non-constant forcing terms $a$.
\begin{proposition}[Perov-type inequality]
\label{prop:perov}
Let $y:[0,T)\to\mathbb{R}_{\geq 0}$ be a continuous function satisfying, for all $t\in[0,T)$,
\begin{equation}
y(t) \;\leq\; a(t) + \int_0^t b(\tau)\,y(\tau)\,\dtau + \int_0^t c(\tau)\,y(\tau)^\gamma\,\dtau,
\label{eq:perov_hyp}
\end{equation}
where $\gamma>1$ is fixed, $a:[0,T)\to\mathbb{R}_{\geq 0}$ is non-decreasing and continuous, and $b,c:[0,T)\to\mathbb{R}_{\geq 0}$ are continuous. Then, for every $t\in[0,T)$ such that $a(t)>0$ and
\begin{equation}
a(t)^{\gamma-1}\,(\gamma-1)\int_0^t c(\tau)\,\exp\left((\gamma-1)\left({\displaystyle\int_0^\tau b(\xi)\,\mathrm{d}\xi}\right)\right)\,\dtau \;<\; 1,
\label{eq:perov_cond}
\end{equation}
the following estimate holds:
\begin{equation}
\label{eq:perov}
y(t) \;\leq\; \frac{a(t)\,\exp\left({\displaystyle\int_0^t b(\tau)\,\dtau}\right)}{\Big(1 - a(t)^{\gamma-1}\,(\gamma-1)\int_0^t c(\tau)\,\exp\left((\gamma-1)\left({\displaystyle\int_0^\tau b(\xi)\,\mathrm{d}\xi}\right)\right)\,\dtau\Big)^{\frac{1}{\gamma-1}}}.
\end{equation}
\end{proposition}
\begin{proof}
The proof follows the steps of \cite[Theorem~3.1]{webb_extensions_2018}, noting that since $a$ is non-decreasing we may replace $a(\tau)$ with $a(t)$ for every $\tau\in[0,t]$, freezing the right-hand side of \eqref{eq:perov_hyp} at the fixed value $\bar{a} := a(t)$ before proceeding with the Bernoulli substitution $u=z^{1-\gamma}$.
\end{proof}

\section{Stability analysis}
\label{sec:stability}
Before deriving the stability results, we define the following energy functionals,
\begin{subequations}
\begin{alignat}{3}
\label{eq:enorm1}
    \eonenorm{\boldsymbol{v}_h}^2 := & \Lpnorm{\boldsymbol{v}_h}{2}^2 
    + \int_{0}^{t} \left( \frac{1}{2}
    \DGnorm{\boldsymbol{v}_h}^2 + 2 \sum_{j=1}^{n_1} c_j^{(1)} 
    \Lpnormsc{v_{jh}}{4}^4\right)\dtau, \quad && \forall \boldsymbol{v}_h \in \Wh, \\
\label{eq:enorm2}
    \etwonorm{\boldsymbol{v}_h}^2
:= & \frac{1}{4} \DGnorm{\boldsymbol{v}_h}^2
 +\sum_{j=1}^{n_1}c_j^{(1)}
\Lpnormsc{v_{jh}}{4}^4
+\frac{1}{2}\int_0^t
\Lpnorm{\dot{\boldsymbol{v}_h}}{2}^2\,\dtau, \quad && \forall \boldsymbol{v}_h \in \Wh.
\end{alignat}
\end{subequations}
To study the stability of system \eqref{eq:semidiscrete}, we start by selecting the first block, associated with $p=1$. In this case, we have to analyze the stability of the system:
\begin{equation*}
    (\dot{\boldsymbol{w}}^{(1)}_h,\boldsymbol{v}^{(1)}_h)+ \mathcal{A}_1(\boldsymbol{w}^{(1)}_h,\boldsymbol{v}^{(1)}_h ) = (\mathbf{L}^{(1,1)} + (\mathbf{Q}^{(1,1,1)} + \mathbb{C}^{(1,1,1,1)} \boldsymbol{w}_h^{(1)}) \boldsymbol{w}_h^{(1)}) \boldsymbol{w}_h^{(1)}  + \boldsymbol{\gamma}^{(1)},\boldsymbol{v}_h^{(1)}),
\end{equation*}
where we underline we do not have any cubic reaction dependent on previous blocks.
\begin{theorem}[Stability result for the first block of populations]
\label{thm:stab_DG_1}
Let $d=2$ and suppose Assumptions \ref{ass:c_structure} and \ref{ass:mesh} hold. Let $\boldsymbol{w}_h(t) \in \Wh$ be the solution of problem \eqref{eq:semidiscrete}. If the stability parameter $\eta_l$ defined in Equation \eqref{eq:penalty_def} for each $l=1,\dots,n$ has been chosen sufficiently large, then the following stability estimates hold:
\begin{equation}
\label{eq:energy_L2}
    \eonenorm{\boldsymbol{w}_{h}^{(1)}(t)}^2 \leq \dfrac{\left(
    \Lpnorm{\boldsymbol{w}_{0h}^{(1)}}{2}^2 + \frac{t}{l_{\infty}} 
    \Lpnorm{\boldsymbol{\gamma}^{(1)}}{2}^2 \right)e^{3l_{\infty}t}}{\left(1-\frac{2q_{\infty}^2\CGN{3}^6(e^{3l_{\infty}t}-1)}{3l_{\infty}}\left(\Lpnorm{\boldsymbol{w}_{0h}^{(1)}}{2}^2 + \frac{t}{l_\infty} \Lpnorm{\boldsymbol{\gamma}^{(1)}}{2}^2\right)\right)} =: \Cen{1}{1}(t).
\end{equation}
and
\begin{equation}
\label{eq:energy_DG}
\etwonorm{\boldsymbol{w}_{h}^{(1)}(t)}^2 \leq 
\left(2\DGnorm{\boldsymbol{w}_{0h}^{(1)}}^2
+
c_\infty
\Lpnorm{\boldsymbol{w}_{0h}^{(1)}}{4}^4
+
t\Lpnorm{\boldsymbol{\gamma}^{(1)}}{2}^2
+
l_\infty\int_0^t
\Cen{1}{1}(\tau)\,\dtau
\right)
e^{ 2q_\infty^2\CGN{4}^{4}
\displaystyle\int_0^t
\Cen{1}{1}(\tau)\,\dtau
},
\end{equation}
for $t \in (0,\hat{t})$ such that the nondegeneracy condition holds:
\begin{equation*}
    \left(e^{3l_\infty \hat{t}}-1\right) \left( \Lpnorm{\boldsymbol{w}_{0h}^{(1)}}{2}^{2} +
\frac{\hat{t}}{l_\infty} \Lpnorm{\boldsymbol{\gamma}^{(1)}}{2}^{2}
\right) < \frac{3l_\infty}{2q_\infty^2 \CGN{3}^{6}}.
\end{equation*}
\end{theorem}
\begin{proof}
\textbf{Proof of equation~\eqref{eq:energy_L2}}. Let us start by taking into account the first block of equations of problem \eqref{eq:general_block_strong}. With this purpose we take as test function $\boldsymbol{v}_h = [\boldsymbol{w}^{(1)}_h, \boldsymbol{0}]^T$ with $\boldsymbol{w}^{(1)}_h = [w_{1h},\dots,w_{n_1h}]^T$ and integrate in time between $0$ and $t$:
\begin{equation*}
    \int_{0}^{t} \left((\dot{\boldsymbol{w}}^{(1)}_h,\boldsymbol{w}^{(1)}_h)+ \mathcal{A}_1(\boldsymbol{w}^{(1)}_h,\boldsymbol{w}^{(1)}_h )\right)\dtau= \int_{0}^{t} \left((\mathbf{L}^{(1,1)} + (\mathbf{Q}^{(1,1,1)} + \mathbb{C}^{(1,1,1,1)} \boldsymbol{w}_h^{(1)}) \boldsymbol{w}_h^{(1)}) \boldsymbol{w}_h^{(1)}  + \boldsymbol{\gamma}^{(1)},\boldsymbol{w}_h^{(1)}\right) \dtau.
\end{equation*}
Then, we can use the $L^\infty$-bounds of the tensors $\mathbf{L}$ and $\mathbf{Q}$ to bound the RHS
\begin{align*}
    ((\mathbf{L}^{(1,1)} 
    + & \mathbf{Q}^{(1,1,1)} \boldsymbol{w}_h^{(1)}) \boldsymbol{w}_h^{(1)}  
    + \boldsymbol{\gamma}^{(1)},\boldsymbol{w}_h^{(1)}) 
    \leq l_{\infty} \Lpnorm{\boldsymbol{w}_{h}^{(1)}}{2}^2 
    + q_{\infty}    \Lpnorm{\boldsymbol{w}_{h}^{(1)}}{3}^3 
    + \Lpnorm{\boldsymbol{\gamma}^{(1)}}{2}\Lpnorm{\boldsymbol{w}_{h}^{(1)}}{2}.
\end{align*}
Then we can integrate by parts $(\dot{\boldsymbol{w}}^{(1)}_h,\boldsymbol{w}^{(1)}_h)$ and use Assumption~\ref{ass:c_structure} and Proposition~\ref{prop:coercivity} to bound from below the LHS
\begin{align*}
     \Lpnorm{\boldsymbol{w}_{h}^{(1)}(t)}{2}^2 + 
     & \int_{0}^{t} \left( 
     \DGnorm{\boldsymbol{w}_{h}^{(1)}(\tau)}^2 + 
     2 \sum_{j=1}^{n_1} c_j^{(1)} 
     \Lpnormsc{w_{jh}^{(1)}(\tau)}{4}^4\right)\dtau \\ \leq & 
     \Lpnorm{\boldsymbol{w}_{0h}^{(1)}}{2}^2 + 
     \int_{0}^{t}\left((2l_{\infty}+\varepsilon_{\gamma}^{-1}) 
     \Lpnorm{\boldsymbol{w}_{h}^{(1)}(\tau)}{2}^2 + 2q_{\infty} 
     \Lpnorm{\boldsymbol{w}_{h}^{(1)}(\tau)}{3}^3 + \varepsilon_{\gamma}
     \Lpnorm{\boldsymbol{\gamma}^{(1)}}{2}^2 \right) \dtau.
\end{align*}
Then, we use the inequality \eqref{eq:gn_adapt}, with $p=3$ and $\alpha=3$. We subtract the DG-norm on the LHS, choosing to simplify $\varepsilon_{\gamma} = 1/l_{\infty}$ and $\varepsilon_{3,3} = 1/\sqrt{2 q_\infty}$. Then we obtain
\begin{align*}
    \eonenorm{\boldsymbol{w}_{h}^{(1)}(t)}^2 = & \Lpnorm{\boldsymbol{w}_{h}^{(1)}(t)}{2}^2 
    + \int_{0}^{t} \left( \frac{1}{2}
    \DGnorm{\boldsymbol{w}_{h}^{(1)}(\tau)}^2 + 2 \sum_{j=1}^{n_1} c_j^{(1)} 
    \Lpnormsc{w_{jh}^{(1)}(\tau)}{4}^4\right)\dtau \\ \leq & \Lpnorm{\boldsymbol{w}_{0h}^{(1)}}{2}^2 + \dfrac{t}{l_{\infty}}\Lpnorm{\boldsymbol{\gamma}^{(1)}}{2}^2 +  3l_{\infty} \int_{0}^{t} \Lpnorm{\boldsymbol{w}_{h}^{(1)}(\tau)}{2}^2 \dtau + 2q_{\infty}^2\CGN{3}^6  \int_{0}^{t}\Lpnorm{\boldsymbol{w}_{h}^{(1)}(\tau)}{2}^4  \dtau,
\end{align*}
where we used the energy functional defined in \eqref{eq:enorm1}. Finally, using the Perov inequality in Proposition~\ref{prop:perov}, we obtain the estimate in equation~\eqref{eq:energy_L2}.
\bigskip
\\
\textbf{Proof of equation~\eqref{eq:energy_DG}}.
Take $\boldsymbol{v}_h=[\dot{\boldsymbol{w}}_h^{(1)},\boldsymbol{0}]^T$ in the first block and integrate over $(0,t)$. Since $\mathcal A_1$ is symmetric,
\begin{equation*}
\int_0^t
\mathcal{A}_1\left(
\boldsymbol{w}_h^{(1)}(\tau),
\dot{\boldsymbol{w}}_h^{(1)}(\tau)
\right)\,\dtau
=
\frac12\mathcal{A}_1\left(
\boldsymbol{w}_h^{(1)}(t),
\boldsymbol{w}_h^{(1)}(t)
\right)
-
\frac12\mathcal{A}_1\left(
\boldsymbol{w}_{0h}^{(1)},
\boldsymbol{w}_{0h}^{(1)}
\right).
\end{equation*}
We estimate the linear and forcing terms by Young's inequality:
\begin{align*}
l_\infty
\Lpnorm{\boldsymbol{w}_h^{(1)}}{2}
\Lpnorm{\dot{\boldsymbol{w}}_h^{(1)}}{2}
&\leq
\frac{1}{4}
\Lpnorm{\dot{\boldsymbol{w}}_h^{(1)}}{2}^2
+
l_\infty^2
\Lpnorm{\boldsymbol{w}_h^{(1)}}{2}^2,
\\
\Lpnorm{\boldsymbol{\gamma}^{(1)}}{2}
\Lpnorm{\dot{\boldsymbol{w}}_h^{(1)}}{2}
&\leq
\frac{1}{4}
\Lpnorm{\dot{\boldsymbol{w}}_h^{(1)}}{2}^2
+
\Lpnorm{\boldsymbol{\gamma}^{(1)}}{2}^2.
\end{align*}
Moreover, Hölder's and Young's inequalities yield
\begin{align*}
\left|
\left(
(\mathbf{Q}^{(1,1,1)}\boldsymbol{w}_h^{(1)})
\boldsymbol{w}_h^{(1)},
\dot{\boldsymbol{w}}_h^{(1)}
\right)
\right|
&\leq
\frac{1}{2}\Lpnorm{\dot{\boldsymbol{w}}_h^{(1)}}{2}^2
+
q_\infty^2
\frac{1}{2}\Lpnorm{\boldsymbol{w}_h^{(1)}}{4}^4
\\
&\leq
\frac{1}{2}\Lpnorm{\dot{\boldsymbol{w}}_h^{(1)}}{2}^2
+
\frac{1}{2}q_\infty^2\CGN{4}^{4}
\DGnorm{\boldsymbol{w}_h^{(1)}}^2
\Lpnorm{\boldsymbol{w}_h^{(1)}}{2}^2.
\end{align*}
Using Assumption~\ref{ass:c_structure}, coercivity, and continuity, we infer
\begin{align*}
\etwonorm{\boldsymbol{w}_{h}^{(1)}(t)}^2
:= \frac{1}{4} \DGnorm{\boldsymbol{w}_{h}^{(1)}(t)}^2
& +\sum_{j=1}^{n_1}c_j^{(1)}
\Lpnormsc{w_{jh}^{(1)}(t)}{4}^4
+\frac12\int_0^t
\Lpnorm{\dot{\boldsymbol{w}}_{h}^{(1)}(\tau)}{2}^2\,\dtau
\\
& \leq
2\DGnorm{\boldsymbol{w}_{0h}^{(1)}}^2
+c_\infty\Lpnorm{\boldsymbol{w}_{0h}^{(1)}}{4}^4
+t\Lpnorm{\boldsymbol{\gamma}^{(1)}}{2}^2
+l_\infty^2\int_0^t
\Lpnorm{\boldsymbol{w}_{h}^{(1)}(\tau)}{2}^2\,\dtau \\ & +\frac{1}{2}q_\infty^2\CGN{4}^{4}
\int_0^t
\DGnorm{\boldsymbol{w}_{h}^{(1)}(\tau)}^2
\Lpnorm{\boldsymbol{w}_{h}^{(1)}(\tau)}{2}^2\,\dtau,
\end{align*}
where we used the energy functional defined in \eqref{eq:enorm2}.
Using $\Lpnorm{\boldsymbol{w}_{h}^{(1)}(\tau)}{2}^2\leq\Cen{1}{1}(\tau)$ together with
$\DGnorm{\boldsymbol{w}_{h}^{(1)}(\tau)}^2\leq4\etwonorm{\boldsymbol{w}_{h}^{(1)}(\tau)}^2$, we obtain
\begin{equation*}
\etwonorm{\boldsymbol{w}_{h}^{(1)}(t)}^2
\leq 
2\DGnorm{\boldsymbol{w}_{0h}^{(1)}}^2
+c_\infty\Lpnorm{\boldsymbol{w}_{0h}^{(1)}}{4}^4
+t\Lpnorm{\boldsymbol{\gamma}^{(1)}}{2}^2
+l_\infty^2\int_0^t\Cen{1}{1}(\tau)\,\dtau +
2q_\infty^2\CGN{4}^{4}
\int_0^t
\Cen{1}{1}(\tau)
\etwonorm{\boldsymbol{w}_{h}^{(1)}(\tau)}^2\,\dtau.
\end{equation*}
Gronwall's inequality yields~\eqref{eq:energy_DG}.
\end{proof}
\begin{remark}
An extension to $d=3$ can be obtained by adapting the Gagliardo--Nirenberg exponents and imposing the corresponding admissibility conditions on the Lebesgue exponents. We do not pursue this extension here. For an application in $d=3$ to the Fisher--Kolmogorov model, we refer to
\cite{corti_discontinuous_2023}.
\end{remark}
\begin{remark}
We emphasize that Theorem~\ref{thm:stab_DG_1} is already a result of independent interest, beyond its role as the base case of the induction leading to Theorem~\ref{thm:stab_DG_p}. Indeed, the first block of the system~\eqref{eq:general_strong} reduces to a general nonlinear reaction–diffusion system with purely quadratic interactions. This general model contains, as particular cases, several classical models in mathematical biology and physical chemistry, such as:
(i) Smoluchowski-type coagulation models without secondary nucleation \cite{fornari_spatially-extended_2020}; (ii) reaction–diffusion Lotka–Volterra systems for competing or predator–prey populations \cite{he_effects_2013}; (iii) mass-action
chemical kinetics systems with bimolecular reactions \cite{smoller_shock_1994}; and (iv) epidemic models with spatial diffusion \cite{lotfi_partial_2014}. Theorem~\ref{thm:stab_DG_1} therefore provides, as a byproduct, an $L^2$- and DG-norm stability result for the semi-discrete PolyDG approximation of this whole class of models.
\end{remark}

\begin{theorem}[Stability result for the $p$-th block of populations]
\label{thm:stab_DG_p}
Let $d=2$, and let Assumptions~\ref{ass:c_structure}
and~\ref{ass:mesh} hold. Let $\boldsymbol{w}_h(t)\in\Wh$ be the
solution to problem~\eqref{eq:semidiscrete}. Assume that the stability
parameters $\eta_l$, defined in~\eqref{eq:penalty_def}, are chosen
sufficiently large. Then, for $t\in(0,\widehat{t}_p)$,
\begin{equation}
\label{eq:energy_L2_p}
\eonenorm{\boldsymbol{w}_{h}^{(p)}(t)}^2
\leq \frac{\left(\Lpnorm{\boldsymbol{w}_{0h}^{(p)}}{2}^2
+ \frac{t}{l_{\infty}} \Lpnorm{\boldsymbol{\gamma}^{(p)}}{2}^2 + K_1^{(p)}(t) \right) e^{\beta^{(p)}(0,t)}}{\left[
1 - 2 \left(
\Lpnorm{\boldsymbol{w}_{0h}^{(p)}}{2}^2
+
\frac{t}{l_{\infty}}
\Lpnorm{\boldsymbol{\gamma}^{(p)}}{2}^2
+
K_1^{(p)}(t)
\right)^2
\displaystyle\int_0^t
\left(
2q_{\infty}^2\CGN{3}^{6}
+
K_3^{(p)}(\tau)
\right)
e^{2\left(\beta^{(p)}(\tau,t)\right)}
\,\mathrm \dtau
\right]^{\frac12}
}
=: \Cen{1}{p}(t),
\end{equation}
where $\beta^{(p)}(t_1,t_2)
:=
3l_{\infty}(t_2-t_1)
+
\int_{t_1}^{t_2} K_2^{(p)}(\tau)\,\mathrm \dtau$ and the $K_j^{(p)}$ terms are defined as in \eqref{eq:K-coefficients}. The estimate holds provided that
\begin{equation*}
2
\left(
\Lpnorm{\boldsymbol{w}_{0h}^{(p)}}{2}^2
+
\frac{t}{l_{\infty}}
\Lpnorm{\boldsymbol{\gamma}^{(p)}}{2}^2
+
K_1^{(p)}(\hat{t}_p)
\right)^2
\int_0^{\hat{t}_p}
\left(
2q_{\infty}^2\CGN{3}^{6}
+
K_3^{(p)}(\tau)
\right)
e^{2\left(\beta^{(p)}(\tau,\hat{t}_p)\right)}
\,\mathrm \dtau
<1.
\end{equation*}
Moreover, the following estimate holds for $t \in (0,t_p^*)$
\begin{equation}
\etwonorm{\boldsymbol{w}_h^{(p)}(t)}^2
\;\leq\;
\frac{a^{(p)}(t)\exp\left(\displaystyle\int_{0}^{t} D^{(p)}_2(\tau)\,\dtau\right)}
{1 - a^{(p)}(t)\displaystyle\int_{0}^{t}\left(6 q_\infty C_{E_4}^4 + D^{(p)}_3(\tau)\right)
\exp\!\left(\displaystyle\int_{0}^{\tau} D^{(p)}_2(\xi)\mathrm{d}\xi\right)\dtau},
\label{eq:energy_DG_p}
\end{equation}
with
\begin{equation}
a^{(p)}(t) := 2\DGnorm{\boldsymbol{w}_{0h}^{(p)}}^2 + c_\infty\Lpnorm{\boldsymbol{w}_{0h}^{(p)}}{4}^4 + 6t\Lpnorm{\boldsymbol{\gamma}^{(p)}}{2}^2
+ \int_{0}^{t} 6 l_\infty \Cen{1}{p}(\tau)\,\dtau + D^{(p)}_1(t),
\end{equation}
provided that the following non-degeneracy condition holds
\begin{equation}
\left(2\DGnorm{\boldsymbol{w}_{0h}^{(p)}}^2 + 6t\Lpnorm{\boldsymbol{\gamma}^{(p)}}{2}^2
+ \int_{0}^{t_p^*} 6 l_\infty \Cen{1}{p}(\tau)\,\dtau + D^{(p)}_1(t_p^*)\right)\int_{0}^{t_p^*}\left(6 q_\infty C_{E_4}^4 + D^{(p)}_3(\tau)\right)
e^{\int_{0}^{\tau} D^{(p)}_2(\xi)\,\mathrm{d}\xi}\dtau \;<\; 1,
\label{eq:nondegeneracy_block_p}
\end{equation}
and with the $D_j^{(p)}$ coefficients as in \eqref{eq:D-coefficients}.
\end{theorem}
\begin{proof}
Now, we prove the induction step on the population $p$, assuming that the estimates for the previous $p-1$ blocks have already been proved. 
\bigskip
\\
\textbf{Proof of equation~\eqref{eq:energy_L2_p}}. First, we take as test function $\boldsymbol{v}_h = [\boldsymbol{0},\boldsymbol{w}^{(p)}_h, \boldsymbol{0}]^T$ and integrate in time between $0$ and $t:$
\begin{align*}
    & \int_{0}^{t} \big((\dot{\boldsymbol{w}}^{(p)}_h(\tau),\boldsymbol{w}^{(p)}_h(\tau)) + \mathcal{A}_p(\boldsymbol{w}^{(p)}_h(\tau),\boldsymbol{w}^{(p)}_h(\tau) )\big)\dtau \\ & \leq \int_{0}^{t} \left((\mathbf{L}^{(p,p)} + (\mathbf{Q}^{(p,p,p)} + \mathbb{C}^{(p,p,p,p)} \boldsymbol{w}_h^{(p)}(\tau)) \boldsymbol{w}_h^{(p)}(\tau)) \boldsymbol{w}_h^{(p)}(\tau)  + \boldsymbol{\gamma}^{(p)},\boldsymbol{w}_h^{(p)}(\tau)\right) \dtau 
    \\ & 
    + \int_{0}^{t} \sum_{q=1}^{p-1} \left(l_\infty
    \Lpnorm{\boldsymbol{w}_h^{(q)}(\tau)}{2}
    \Lpnorm{\boldsymbol{w}_h^{(p)}(\tau)}{2} + q_\infty
    \Lpnorm{\boldsymbol{w}_h^{(q)}(\tau)}{3}
    \Lpnorm{\boldsymbol{w}_h^{(p)}(\tau)}{3}^2 + c_\infty
    \Lpnorm{\boldsymbol{w}_h^{(q)}(\tau)}{4}
    \Lpnorm{\boldsymbol{w}_h^{(p)}(\tau)}{4}^3 \right) \dtau 
    \\ & + \int_{0}^{t}\sum_{q,s=1}^{p-1} \left(q_\infty
    \Lpnorm{\boldsymbol{w}_h^{(s)}(\tau)}{3}
    \Lpnorm{\boldsymbol{w}_h^{(q)}(\tau)}{3}
    \Lpnorm{\boldsymbol{w}_h^{(p)}(\tau)}{3} +c_\infty
    \Lpnorm{\boldsymbol{w}_h^{(s)}(\tau)}{4}
    \Lpnorm{\boldsymbol{w}_h^{(q)}(\tau)}{4}
    \Lpnorm{\boldsymbol{w}_h^{(p)}(\tau)}{4}^2 \right) \dtau
    \\ & + \int_{0}^{t}\sum_{q,s,r=1}^{p-1} \left(c_\infty 
    \Lpnorm{\boldsymbol{w}_h^{(r)}(\tau)}{4}
    \Lpnorm{\boldsymbol{w}_h^{(s)}(\tau)}{4}
    \Lpnorm{\boldsymbol{w}_h^{(q)}(\tau)}{4}
    \Lpnorm{\boldsymbol{w}_h^{(p)}(\tau)}{4}\right) \dtau.
\end{align*}
Using the previous steps of the induction and repeating the steps of the proof of equation \eqref{eq:energy_L2}, we obtain
\begin{align*}
   & \Lpnorm{\boldsymbol{w}_{h}^{(p)}(t)}{2}^2 + 
      \int_{0}^{t} \left( 
     \frac{1}{2} \DGnorm{\boldsymbol{w}_{h}^{(p)}(\tau)}^2 + 
     2 \sum_{j=1}^{n_p} c_j^{(p)} 
     \Lpnormsc{w_{jh}^{(p)}(\tau)}{4}^4\right)\dtau \\  & \qquad \leq \Lpnorm{\boldsymbol{w}_{0h}^{(p)}}{2}^2 + \dfrac{t}{l_{\infty}}\Lpnorm{\boldsymbol{\gamma}^{(p)}}{2}^2 +  3l_{\infty} \int_{0}^{t} \Lpnorm{\boldsymbol{w}_{h}^{(p)}(\tau)}{2}^2 \dtau + 2q_{\infty}^2\CGN{3}^6  \int_{0}^{t}\Lpnorm{\boldsymbol{w}_{h}^{(p)}(\tau)}{2}^4  \dtau
    \\ & \qquad + \int_{0}^{t} \sum_{q=1}^{p-1} \left(l_\infty \sqrt{\Cen{1}{q}(\tau)}
    \Lpnorm{\boldsymbol{w}_h^{(p)}(\tau)}{2} + 2q_\infty \sqrt{\Cen{2}{q}(\tau)} C_{E_3} 
    \Lpnorm{\boldsymbol{w}_h^{(p)}(\tau)}{3}^2 + 2c_\infty \sqrt{\Cen{2}{q}(\tau)} C_{E_4}
    \Lpnorm{\boldsymbol{w}_h^{(p)}(\tau)}{4}^3 \right) \dtau 
    \\ & \qquad + \int_{0}^{t}\sum_{q,s=1}^{p-1} \left(4q_\infty \sqrt{\Cen{2}{q}(\tau) \Cen{2}{s}(\tau)} C_{E_3}^2 
    \Lpnorm{\boldsymbol{w}_h^{(p)}(\tau)}{3} + 4c_\infty \sqrt{\Cen{2}{q}(\tau) \Cen{2}{s}(\tau)} C_{E_4}^2  
    \Lpnorm{\boldsymbol{w}_h^{(p)}(\tau)}{4}^2 \right) \dtau
    \\ & \qquad + \int_{0}^{t}\left(\sum_{q,s,r=1}^{p-1} 8c_\infty \sqrt{\Cen{2}{q}(\tau) \Cen{2}{s}(\tau) \Cen{2}{r}(\tau)} C_{E_4}^3  \Lpnorm{\boldsymbol{w}_h^{(p)}(\tau)}{4} \right) \dtau,
\end{align*}
where $C_{E_k}$ is the Sobolev-Poincaré embedding \cite[Theorem~6.5]{di_pietro_hybrid_2020} constant for the $L^k(\Omega)$ space. By applying
the inequality \eqref{eq:gn_adapt} to the
terms involving the $L^3(\Omega)$ and $L^4(\Omega)$ norms of
$\boldsymbol{w}_h^{(p)}$, with the parameters
$\varepsilon_{p,\alpha}$ chosen so that each contribution involving
$\DGnorm{\boldsymbol{w}_h^{(p)}}^2$ is equal to
$\frac{1}{10}\DGnorm{\boldsymbol{w}_h^{(p)}}^2$, we obtain
\begin{align}
\label{eq:partial_1}
    \eonenorm{{\boldsymbol{w}_{h}^{(p)}(t)}}^2 \leq & \Lpnorm{\boldsymbol{w}_{0h}^{(p)}}{2}^2 + \dfrac{t}{l_{\infty}}\Lpnorm{\boldsymbol{\gamma}^{(p)}}{2}^2 +  3l_{\infty} \int_{0}^{t} \Lpnorm{\boldsymbol{w}_{h}^{(p)}(\tau)}{2}^2 \dtau + 2q_{\infty}^2\CGN{3}^6  \int_{0}^{t}\Lpnorm{\boldsymbol{w}_{h}^{(p)}(\tau)}{2}^4  \dtau
    \\ + & \int_{0}^{t} \left(
    A_{2,1}^{(p)}(\tau) \Lpnorm{\boldsymbol{w}_h^{(p)}(\tau)}{2} + 
    (A_{3,2}^{(p)}(\tau)+A_{4,2}^{(p)}(\tau))\Lpnorm{\boldsymbol{w}_h^{(p)}(\tau)}{2}^{2} + A_{4,3}^{(p)}(\tau) 
    \Lpnorm{\boldsymbol{w}_h^{(p)}(\tau)}{2}^{6} \right) \dtau 
    \\ + & \int_{0}^{t} \left( A_{3,1}^{(p)}(\tau)
\Lpnorm{\boldsymbol{w}_h^{(p)}(\tau)}{2}^{\frac{4}{5}} + A_{4,1}^{(p)}(\tau) \Lpnorm{\boldsymbol{w}_h^{(p)}(\tau)}{2}^{\frac{2}{3}} \right) \dtau.
\end{align}
Here, the non-negative functions $A_{k,\alpha}^{(p)}$ are defined as in Equation \eqref{eq:A-coefficients}. Now, to handle the subquadratic terms, we notice that $x^\beta\leq (1-\frac{\beta}{2})(\frac{\beta}{2})^{\frac{\beta}{2-\beta}}+x^2$ for every $\beta\in(0,1]$. On the other hand we observe that $x^4 \leq 4/27 + x^6$ for all $x\in \mathbb{R}$. Using them in \eqref{eq:partial_1}, we derive
\begin{align*}
    \eonenorm{{\boldsymbol{w}_{h}^{(p)}(t)}}^2 \leq & \Lpnorm{\boldsymbol{w}_{0h}^{(p)}}{2}^2 + \dfrac{t}{l_{\infty}}\Lpnorm{\boldsymbol{\gamma}^{(p)}}{2}^2 + K^{(p)}_1(t) + \int_{0}^{t} (3l_\infty + K^{(p)}_2(\tau))\Lpnorm{\boldsymbol{w}_h^{(p)}(\tau)}{2}^{2} \dtau
    \\ + & \int_{0}^{t} (2q_{\infty}^2\CGN{3}^6+K^{(p)}_3(\tau)) 
    \Lpnorm{\boldsymbol{w}_h^{(p)}(\tau)}{2}^{6} \dtau,
\end{align*}
where, we define the functions $K^{(p)}_1(t)$, $K^{(p)}_2(t)$, $K^{(p)}_3(t)$ as in \eqref{eq:K-coefficients} such that they collect the quantities connected with the interaction of the $p$-th population with the previous blocks. The thesis follows from the application of Perov's inequality in Proposition \ref{prop:perov}.

\textbf{Proof of equation~\eqref{eq:energy_DG_p}}.
Take $\boldsymbol{v}_h=[\boldsymbol{0},\dot{\boldsymbol{w}}_h^{(p)},\boldsymbol{0}]^T$ in the $p$-th block and integrate over $(0,t)$. Since $\mathcal A_p$ is symmetric,
\begin{equation*}
\int_0^t \mathcal{A}_p\left(\boldsymbol{w}_h^{(p)}(\tau), \dot{\boldsymbol{w}}_h^{(p)}(\tau)\right)\,\dtau
= \frac{1}{2} \mathcal{A}_p\left(\boldsymbol{w}_h^{(p)}(t),
\boldsymbol{w}_h^{(p)}(t)\right)
- \frac{1}{2} \mathcal{A}_p\left(\boldsymbol{w}_{0h}^{(p)},
\boldsymbol{w}_{0h}^{(p)}\right).
\end{equation*}
Proceeding as in the proof of the first block and using the H\"older and Young's inequalities on the reaction terms we obtain 
\begin{align*}
    \int_{0}^{t} & \frac{1}{2}\Lpnorm{\dot{\boldsymbol{w}}^{(p)}_h(\tau)}{2}^2 \dtau + \frac{1}{2}\mathcal{A}_p\left(\boldsymbol{w}_h^{(p)}(t),
\boldsymbol{w}_h^{(p)}(t)\right) + \sum_{j=1}^{n_p}c_j^{(p)}
\Lpnormsc{w_{jh}^{(p)}(t)}{4}^4 \\ & \leq \frac{1}{2}\mathcal{A}_p\left(\boldsymbol{w}_{0h}^{(p)},
\boldsymbol{w}_{0h}^{(p)}\right) + \int_{0}^{t} \left( 6 l_\infty
    \Lpnorm{\boldsymbol{w}_h^{(p)}(\tau)}{2}^2 + 6 q_\infty
    \Lpnorm{\boldsymbol{w}_h^{(p)}(\tau)}{4}^4 + 6
    \Lpnorm{\boldsymbol{\gamma}^{(p)}}{2}^2 \right) \dtau 
    \\ & 
    + \int_{0}^{t} \sum_{q=1}^{p-1} \left( 6l_\infty^2
    \Lpnorm{\boldsymbol{w}_h^{(q)}(\tau)}{2}^2 + 6 q_\infty^2
    \Lpnorm{\boldsymbol{w}_h^{(q)}(\tau)}{4}^2
    \Lpnorm{\boldsymbol{w}_h^{(p)}(\tau)}{4}^2 + 6 c_\infty^2
    \Lpnorm{\boldsymbol{w}_h^{(q)}(\tau)}{4}^2
    \Lpnorm{\boldsymbol{w}_h^{(p)}(\tau)}{8}^4 \right) \dtau 
    \\ & + \int_{0}^{t}\sum_{q,s=1}^{p-1} \left(4 q_\infty^2
    \Lpnorm{\boldsymbol{w}_h^{(s)}(\tau)}{4}^2
    \Lpnorm{\boldsymbol{w}_h^{(q)}(\tau)}{4}^2
     + 4 c_\infty^2
    \Lpnorm{\boldsymbol{w}_h^{(s)}(\tau)}{8}^2
    \Lpnorm{\boldsymbol{w}_h^{(q)}(\tau)}{8}^2
    \Lpnorm{\boldsymbol{w}_h^{(p)}(\tau)}{4}^2 \right) \dtau
    \\ & + \int_{0}^{t}\sum_{q,s,r=1}^{p-1} 2 c_\infty^2 
    \Lpnorm{\boldsymbol{w}_h^{(r)}(\tau)}{6}^2
    \Lpnorm{\boldsymbol{w}_h^{(s)}(\tau)}{6}^2
    \Lpnorm{\boldsymbol{w}_h^{(q)}(\tau)}{6}^2  \dtau.
\end{align*}
Using Assumption~\ref{ass:c_structure}, coercivity, and continuity, the estimates of the previous $p-1$ blocks, and the Sobolev-Poincaré inequality \cite[Theorem~6.5]{di_pietro_hybrid_2020}, we derive
\begin{align*}
   \etwonorm{\boldsymbol{w}_h^{(p)}(t)}^2 & \leq 2
    \DGnorm{\boldsymbol{w}_{0h}^{(p)}}^2 +c_\infty\Lpnorm{\boldsymbol{w}_{0h}^{(p)}}{4}^4 + 6 t
    \Lpnorm{\boldsymbol{\gamma}^{(p)}}{2}^2 + \int_{0}^{t} 6 l_\infty \Cen{1}{p}(\tau) \dtau + D^{(p)}_1(t) 
    \\ & + \int_{0}^t D^{(p)}_2(\tau) \etwonorm{\boldsymbol{w}_h^{(p)}(\tau)}^2 \dtau + \int_{0}^{t} \left( 6 q_\infty C_{E_4}^2 + D^{(p)}_3(\tau) \right) \etwonorm{\boldsymbol{w}_h^{(p)}(\tau)}^4 \dtau.
\end{align*}
where we exploit the definition of $\etwonorm{\cdot}$ and the functions $D^{(p)}_1(t)$, $D^{(p)}_2(t)$, $D^{(p)}_3(t)$ are defined as in \eqref{eq:D-coefficients}. The thesis follows from the application of the Perov inequality in Proposition~\ref{prop:perov}.
\end{proof}
\begin{remark}
The proof of Theorem~\ref{thm:stab_DG_p} cannot be trivially extended to the three-dimensional case $d=3$. The obstruction arises because in three
dimensions the critical Sobolev exponent is $6$, so that $H^1(\Omega)\hookrightarrow L^q(\Omega)$
only holds for $q\in[2,6]$, and $H^1(\Omega)\not\hookrightarrow L^8(\Omega)$.
\end{remark}

\begin{remark}[Dependence on previous blocks]
We emphasize that both stability estimates in Theorem~\ref{thm:stab_DG_p} for the $p$-th block
are not self-contained, but depend explicitly on the stability estimates of \emph{all} the previous $p-1$
blocks. The same cumulative dependence affects the time of degeneracy of the estimates. Both the definitions of
$\hat{t}_p$ and $t_p^*$ involve the functions $K_i^{(p)}$ and $D_i^{(p)}$, then the admissible time interval for the $p$-th block is, in general, no larger than that of every previous block. This reflects the hierarchical
nature of the block structure introduced in Assumption~\ref{ass:c_structure}: instabilities or blow-up phenomena occurring in an early block of the hierarchy necessarily propagate to, and restrict the validity horizon of, all higher blocks that interact with it, then we have that $\min_{p=1,\dots,p_M}\{\hat{t_p},t^*_p\} = t^*_{p_M}$.
\end{remark}
\begin{remark}[Existence and uniqueness of the discrete solution]
Since $\Wh$ has finite dimension, problem~\eqref{eq:semidiscrete} is equivalent, upon choosing a basis, to a system of ODEs with polynomial (hence locally Lipschitz) right-hand side. By the Cauchy--Lipschitz theorem, this system admits a unique local-in-time solution $\boldsymbol{w}_h \in C^1([0,\tau);\mathbf{W}_h^{\mathrm{DG}})$ for some $\tau>0$, which can be extended to a maximal interval
$[0,T_{\max})$. The a priori estimates of Theorems~\ref{thm:stab_DG_1} and~\ref{thm:stab_DG_p} rule out this blow-up scenario on
the respective intervals of validity, so that the unique discrete solution $\boldsymbol{w}_h(t)$ is guaranteed to exist on $(0,t^*_{p_M})$.
\end{remark}
\section{Error analysis of the semi-discrete formulation}
\label{sec:error_analysis}
In this section, we discuss the a priori error estimates for the solution of the semi-discrete problem in \eqref{eq:semidiscrete}. First of all, we introduce the additional norm defined as
\begin{equation}
\label{eq:triplenorm}
    \TDGnorm{\boldsymbol{u}}^2 = \DGnorm{\boldsymbol{u}}^2 + \|\eta^{-1/2}\avg{\mathbb{D}:\nabla_h\boldsymbol{u}}\|^2_{\mathbf{L}^2(\facesinternal\cup\facesD)}, \qquad \forall \boldsymbol{u}\in \mathbf{H}^2(\partition).
\end{equation}
To extend the bilinear forms of \eqref{eq:bilinear_form_def} to the space of continuous solutions we need further regularity requirements. We assume element-wise $H^2$-regularity of the solution concentration together with the continuity of the flow across the interfaces $F\in\facesinternal$ for all time $t\in(0,T]$. Moreover, the following continuity result holds.
\begin{proposition}[Continuity]
\label{prop:continuity_2}
The bilinear form $\mathcal{A}$, defined in Eq.~\eqref{eq:bilinear_form_def}, is continuous, namely it satisfies
\begin{equation}
\label{eq:continuity_2}
\left|\mathcal{A}(\boldsymbol{u},\boldsymbol{v}_h)\right|
\leq 2\TDGnorm{\boldsymbol{u}}\|\boldsymbol{v}_h\|_\mathrm{DG}
\qquad \forall\,\boldsymbol{u}\in H^2(\partition),\boldsymbol{v}_h\in\Wh.
\end{equation}
\end{proposition}
\begin{proof}
The proof is the same of Proposition \ref{prop:continuity}, but neglecting the inverse trace inequality on the $\boldsymbol{u}$, for which it does not hold in principle.
\end{proof}
Finally, we introduce the interpolant $\boldsymbol{w}_\mathrm{I}\in\Wh$ of the solution of the continuous formulation \eqref{eq:general_weak} defined as in \cite{babuska_p_1994,di_pietro_hybrid_2020}.
\begin{proposition}[Interpolation result \cite{babuska_p_1994}]
\label{prop:interpolant}
Let $K\in\partition$ and assumption \ref{ass:mesh} be fulfilled. Let $\boldsymbol{u}_\mathrm{I}\in\Wh$ be the interpolant of a function $\boldsymbol{u}\in H^\nu(\partition)$ for $\nu\geq 2$. Then the following estimate holds provided that $\tilde{\eta}$ is large enough:
\begin{equation}
    \TDGnorm{\boldsymbol{u}-\boldsymbol{u}_\mathrm{I}}^2 \leq \sum_{K\in\partition} C_\mathrm{I} \frac{h_K^{2\min\{\ell+1,\nu\}-2}}{\ell^{\nu-2}} \|\boldsymbol{u}\|^2_{\mathbf{H}^\nu(K)}
\end{equation}
\end{proposition}
In this section, we assume that $\boldsymbol{w}_{0h} = \boldsymbol{w}_\mathrm{I}(0)\in \Wh$. Now, we are ready to state the a priori estimate.
\begin{theorem}[A priori error estimate]
\label{thm:apriori}
Let us consider problem \eqref{eq:general_weak} and its solution $\boldsymbol{w}\in \mathbf{C}^1((0,T]; \mathbf{H}^\nu(\Omega))$ for $\nu \geq 2$ and let Assumptions \ref{ass:c_structure} and \ref{ass:mesh} be fulfilled. 
Let us assume further regularity on the initial condition $\boldsymbol{w}_0\in \mathbf{H}^1(\Omega)$. 
For a sufficiently large penalty parameter $\tilde{\eta}$, let $\boldsymbol{w}_h$ be the solution of \eqref{eq:semidiscrete} for any $t\in (0,t^*_{p_M}]$, the following estimate holds: 
\begin{align*}
\|\boldsymbol{w}(t)-\boldsymbol{w}_h(t)\|_{\mathbf{L}^2(\Omega)}^2
& + \frac{1}{2}\int_0^t \TDGnorm{\boldsymbol{w}(\tau)-\boldsymbol{w}_h(\tau)}^2\,\dtau
\;\\ & \lesssim
\sum_{K\in\partition} h_K^{2\min\{\ell+1,\nu\}-2} \left(\|\boldsymbol{w}(t)\|^2_{\mathbf{H}^\nu(K)} + \int_0^t \|\boldsymbol{w}(\tau)\|^2_{\mathbf{H}^\nu(K)} \dtau +  \int_0^t \|\dot{\boldsymbol{w}}(\tau)\|^2_{\mathbf{H}^\nu(K)}\right) \exp\left(\displaystyle\int_0^t \alpha(\tau) \dtau\right),
\end{align*}
where the hidden constant and $\alpha(\tau)$ may depend on $\boldsymbol{w}$, $\boldsymbol{w}_{0h}$, $\boldsymbol{\gamma}$, $\Omega$, $\ell$ but is independent of the discretization parameter $h$.
\end{theorem}
\begin{proof}
First, we subtract equation \eqref{eq:semidiscrete} from \eqref{eq:general_weak}, using $\boldsymbol{v} = \boldsymbol{v_h}\in \Wh$.
\begin{align*}
    (\dot{\boldsymbol{w}}-\dot{\boldsymbol{w}}_h,\boldsymbol{v}_h) + \mathcal{A}(\boldsymbol{w}-\boldsymbol{w}_h,\boldsymbol{v}_h)  = & \,(\mathbf{L}(\boldsymbol{w}-\boldsymbol{w}_h),\boldsymbol{v}_h) + \left((\mathbf{Q} \boldsymbol{w}) \boldsymbol{w} - (\mathbf{Q} \boldsymbol{w}_h) \boldsymbol{w}_h,\boldsymbol{v}_h\right)  \\ + & \,  \left(((\mathbb{C}\boldsymbol{w}) \boldsymbol{w}) \boldsymbol{w}-((\mathbb{C} \boldsymbol{w}_h) \boldsymbol{w}_h) \boldsymbol{w}_h,\boldsymbol{v}_h\right) 
\end{align*}
Then we decompose the difference $\boldsymbol{w}-\boldsymbol{w}_h = (\boldsymbol{w}-\boldsymbol{w}_\mathrm{I})+(\boldsymbol{w}_\mathrm{I}-\boldsymbol{w}_h) =: -\boldsymbol{e}_\mathrm{I} + \boldsymbol{e}_h$, where $\boldsymbol{w}_\mathrm{I}$ is the interpolant defined as in Proposition \ref{prop:interpolant}. Moreover, we choose $\boldsymbol{v}_h = \boldsymbol{e}_h$. Then, we obtain
\begin{equation}
\label{eq:step_a_priori}
\begin{aligned}
    (\dot{\boldsymbol{e}}_h,\boldsymbol{e}_h) + \mathcal{A}(\boldsymbol{e}_h,\boldsymbol{e}_h) \leq & (\dot{\boldsymbol{e}}_\mathrm{I},\boldsymbol{e}_h) +\mathcal{A}(\boldsymbol{e}_\mathrm{I},\boldsymbol{e}_h) +l_\infty \Lpnorm{\boldsymbol{e}_h}{2}^2 - l_\infty(\boldsymbol{e}_\mathrm{I},\boldsymbol{e}_h) \\ + & \left((\mathbf{Q} \boldsymbol{w}) \boldsymbol{w} -  (\mathbf{Q} \boldsymbol{w}_h) \boldsymbol{w}_h,\boldsymbol{e}_h\right) + \left(((\mathbb{C}\boldsymbol{w}) \boldsymbol{w}) \boldsymbol{w}-((\mathbb{C} \boldsymbol{w}_h) \boldsymbol{w}_h) \boldsymbol{w}_h,\boldsymbol{v}_h\right). 
\end{aligned}
\end{equation}
Then we rewrite the nonlinear terms using the following algebraic relations to neglect the dependence on $\boldsymbol{w}$ where we exploit the $L^\infty(\Omega)$-bounds of the parameters.
\begin{align*}
|(\mathbf{Q} \boldsymbol{w}) \boldsymbol{w} -  (\mathbf{Q} \boldsymbol{w}_h) \boldsymbol{w}_h| & \leq q_\infty(|\boldsymbol{w}| + |\boldsymbol{w}_h|)|\boldsymbol{w} -\boldsymbol{w}_h| = q_\infty(|\boldsymbol{w}| + |\boldsymbol{w}_h|) |\boldsymbol{e}_h -\boldsymbol{e}_\mathrm{I}|,\\
| ((\mathbb{C}\boldsymbol{w}) \boldsymbol{w}) \boldsymbol{w}-((\mathbb{C} \boldsymbol{w}_h) \boldsymbol{w}_h) \boldsymbol{w}_h | & \leq c_\infty(|\boldsymbol{w}|^2 + |\boldsymbol{w}_h||\boldsymbol{w}| +  |\boldsymbol{w}_h|^2) |\boldsymbol{w} -\boldsymbol{w}_h| = c_\infty(|\boldsymbol{w}|^2 + |\boldsymbol{w}_h| |\boldsymbol{w}| +  |\boldsymbol{w}_h|^2) |\boldsymbol{e}_h -\boldsymbol{e}_\mathrm{I}|,
\end{align*}
Concerning the nonlinear terms, we can bound the terms using H\"older inequality, Young inequality and Equation \eqref{eq:gn_adapt}:
\begin{align*}
    ((q_\infty(|\boldsymbol{w}| + & |\boldsymbol{w}_h|)+c_\infty(|\boldsymbol{w}|^2 + |\boldsymbol{w}_h||\boldsymbol{w}| +  |\boldsymbol{w}_h|^2))(\boldsymbol{e}_h -\boldsymbol{e}_\mathrm{I}),\boldsymbol{e}_h) \\ \leq & 
    \left(q_\infty(\Lpnorm{\boldsymbol{w}}{2} +
    \Lpnorm{\boldsymbol{w}_h}{2}) + 2c_\infty(\Lpnorm{\boldsymbol{w}}{4}^2 +
    \Lpnorm{\boldsymbol{w}_h}{4}^2)\right)
    \Lpnorm{\boldsymbol{e}_h}{4}^2 \\ & +
    \left(q_\infty(\Lpnorm{\boldsymbol{w}}{4} +
    \Lpnorm{\boldsymbol{w}_h}{4}) + 2c_\infty(\Lpnorm{\boldsymbol{w}}{8}^2 +
    \Lpnorm{\boldsymbol{w}_h}{8}^2)\right)
    \Lpnorm{\boldsymbol{e}_\mathrm{I}}{4}
    \Lpnorm{\boldsymbol{e}_h}{2} \\ &
    \leq 
    \frac{1}{4}\,\|\boldsymbol{e}_h\|_{\mathrm{DG}}^2 +
    C_{\mathrm{GN}_4}^{4} \left(q_\infty \Lpnorm{\boldsymbol{w}}{2} + q_\infty \sqrt{C_{e_1}} + 2c_\infty \Lpnorm{\boldsymbol{w}}{4}^2 +
    8c_\infty C_{E_4} C_{e_2} \right)^2\,\|\boldsymbol{e}_h\|_{\mathbf{L}^2(\Omega)}^2 \\ & +
    \frac{1}{2}\left(q_\infty \Lpnorm{\boldsymbol{w}}{4} + 4 q_\infty \sqrt{C_{E_4} C_{e_2}} + 2 c_\infty \Lpnorm{\boldsymbol{w}}{8}^2 +
    8c_\infty C_{E_8} C_{e_2} \right)^2
    \Lpnorm{\boldsymbol{e}_\mathrm{I}}{4}^2 + 
    \frac{1}{2}\Lpnorm{\boldsymbol{e}_h}{2}^2,
\end{align*}
where in the last step we used the stability estimates in Theorems \ref{thm:stab_DG_1} and \ref{thm:stab_DG_p} and the Sobolev-Poincaré inequality in \cite{di_pietro_hybrid_2020}. 
The result is then applied to \eqref{eq:step_a_priori} integrated between $(0,t)$ and using the estimates in Proposition \ref{prop:coercivity} and \ref{prop:continuity_2}:
\begin{align*}
    &\Lpnorm{\boldsymbol{e}_h(t)}{2}^2 + \frac{1}{8} \int_0^t 
    \DGnorm{\boldsymbol{e}_h(\tau)}^2 \dtau \leq \int_0^t \left(\frac{1}{2} \Lpnorm{\dot{\boldsymbol{e}}_\mathrm{I}(\tau)}{2}^2 + 8 \TDGnorm{\boldsymbol{e}_\mathrm{I}(\tau)}^2 \right) \dtau \\ & \quad +
    \frac{1}{2} \int_0^t \left(l_\infty+ q_\infty \Lpnorm{\boldsymbol{w}(\tau)}{4} + 4 q_\infty \sqrt{C_{E_4} C_{e_2}(\tau)} + 2 c_\infty \Lpnorm{\boldsymbol{w}(\tau)}{8}^2 +
    8c_\infty C_{E_8} C_{e_2}(\tau) \right)^2
    \Lpnorm{\boldsymbol{e}_\mathrm{I}(\tau)}{4}^2 \dtau
    \\ & \quad + \int_0^t \left(1 + \frac{3 l_\infty}{2} + C_{\mathrm{GN}_4}^{4} \left(q_\infty \Lpnorm{\boldsymbol{w}(\tau)}{2} + q_\infty \sqrt{C_{e_1}(\tau)} + 2c_\infty \Lpnorm{\boldsymbol{w}(\tau)}{4}^2 +
    8c_\infty C_{E_4} C_{e_2}(\tau) \right)^2 \right) \Lpnorm{\boldsymbol{e}_h(\tau)}{2}^2 \dtau.
\end{align*}
By Grönwall's inequality, we obtain
\begin{align*}
\|\boldsymbol{e}_h(t)\|_{\mathbf{L}^2(\Omega)}^2
& + \int_0^t \DGnorm{\boldsymbol{e}_h(\tau)}^2\,\dtau
\;\leq\;
\Bigg[
 \int_0^t \left(4
\Lpnorm{\dot{\boldsymbol{e}}_\mathrm{I}(\tau)}{2}^2
+ 24\TDGnorm{\boldsymbol{e}_\mathrm{I}(\tau)}^2\right)\dtau \\
& + \int_0^t 4 \left(q_\infty \Lpnorm{\boldsymbol{w}(\tau)}{4}
+ 4 q_\infty \sqrt{C_{E_4} C_{e_2}(\tau)}
+ 2c_\infty \Lpnorm{\boldsymbol{w}(\tau)}{8}^2
+ 8c_\infty C_{E_8} C_{e_2}(\tau)\right)^2
\Lpnorm{\boldsymbol{e}_\mathrm{I}(\tau)}{4}^2\,\dtau
\Bigg] \\
& \times \exp\!\left(\int_0^t \left(1+\frac{3l_\infty}{2}
+ C_{\mathrm{GN}_4}^{4}\left(q_\infty \Lpnorm{\boldsymbol{w}(\tau)}{2}
+ q_\infty \sqrt{C_{e_1}(\tau)}
+ 2c_\infty \Lpnorm{\boldsymbol{w}(\tau)}{4}^2
+ 8c_\infty C_{E_4} C_{e_2}(\tau)\right)^2\right)\dtau\right).
\end{align*}
Exploiting the definition 
\eqref{eq:triplenorm} for which it holds $\TDGnorm{\boldsymbol{e}_h} \leq 2\TDGnorm{\boldsymbol{e}_h}$ \cite{di_pietro_mathematical_2012}, and using the triangular inequality on $\boldsymbol{w}-\boldsymbol{w}_h$ we obtain:
\begin{align*}
\|\boldsymbol{w}(t)-\boldsymbol{w}_h(t)\|_{\mathbf{L}^2(\Omega)}^2
& + \frac{1}{2}\int_0^t \TDGnorm{\boldsymbol{w}(\tau)-\boldsymbol{w}_h(\tau)}^2\,\dtau
\;\leq\;
\Lpnorm{\boldsymbol{e}_\mathrm{I}(t)}{2}^2 + \int_0^t \TDGnorm{\boldsymbol{e}_\mathrm{I}(\tau)}^2\dtau \\ & + \Bigg[
4 \int_0^t \Lpnorm{\dot{\boldsymbol{e}}_\mathrm{I}(\tau)}{2}^2 \dtau + 24 \int_0^t \TDGnorm{\boldsymbol{e}_\mathrm{I}(\tau)}^2\dtau + 4 \int_0^t \kappa(\tau) \Lpnorm{\boldsymbol{e}_\mathrm{I}(\tau)}{4}^2\,\dtau
\Bigg] \exp\left(\displaystyle\int_0^t \alpha(\tau) \dtau\right).
\end{align*}
The thesis follows from the interpolation estimate in Proposition \ref{prop:interpolant}.
\end{proof}
\begin{remark}
In the case of a purely quadratic nonlinearity ($\mathbb{C} = \boldsymbol{0}$), Theorem~\ref{thm:apriori} not only extends to the vector-valued setting the a priori error estimate result for the Fisher--Kolmogorov equation proved in~\cite[Theorem~2]{corti_discontinuous_2023} for the case $d=2$, but also improves it by removing the structural relations between the reaction terms and the diffusion coercivity constant.
\end{remark}
\section{Time discretization: fully-discrete formulation}
\label{sec:IMEX_RK}
To derive the algebraic formulation of the system we introduce a suitable basis
$(\varphi_{jk})_{k=1}^{N_h}$ for each scalar component of the space $\Wh$ such
that $n\,N_h = \dim(\Wh)$. Then, we construct a vectorial basis
$(\boldsymbol{\Phi}_{k})_{k=1}^{n\,N_h}$ such that
\[
\boldsymbol{\Phi}_{k} = [\varphi_{k},0,\dots,0]^{\top}, \quad k=1,\dots,N_h,
\]
\[
\boldsymbol{\Phi}_{N_h+k} = [0,\varphi_{k},0,\dots,0]^{\top}, \quad k=1,\dots,N_h,
\]
and so on, so that each block of $N_h$ basis functions corresponds to one
population.

Then we rewrite the semi-discrete solution as
\begin{equation*}
   \boldsymbol{w}_{h}(\boldsymbol{x},t)
   = \sum_{k=1}^{n\,N_h} \boldsymbol{W}_{k}(t)\,\boldsymbol{\Phi}_{k}(\boldsymbol{x}),
\end{equation*}
where $\boldsymbol{W}(t) \in \mathbb{R}^{n\,N_h}$ is the corresponding vector
of expansion coefficients written in terms of the chosen basis. Moreover, we
define the matrices and vectors, for $i,j=1,\dots,n\,N_h$,
\begin{alignat*}{3}
    [\mathbf{M}]_{ij} =& \,(\boldsymbol{\Phi}_{j},\boldsymbol{\Phi}_{i}),
    && \qquad \text{(mass matrix),} \\
    [\mathbf{K}]_{ij} =& \,\mathcal{A}(\boldsymbol{\Phi}_{j}, \boldsymbol{\Phi}_{i}),
    && \qquad \text{(stiffness matrix),} \\
    [\boldsymbol{G}(\boldsymbol{W}(t))]_{j} =&\,(\widehat{\boldsymbol{G}}(\boldsymbol{w}_{h}(t)),
    \boldsymbol{\Phi}_{j}),
    && \qquad \text{(nonlinear reaction vector),}\\
    [\boldsymbol{\Gamma}]_{j} =&\,(\boldsymbol{\gamma}, \boldsymbol{\Phi}_{j}),
    && \qquad \text{(forcing vector).}
\end{alignat*}
The resulting semi-discrete problem can be written in the abstract form
\begin{equation}\label{eq:algebraic}
    \begin{dcases}
    \mathbf{M}\dot{\boldsymbol{W}}(t) + \mathbf{K}\boldsymbol{W}(t)
    = \boldsymbol{G}(\boldsymbol{W}(t)) + \boldsymbol{\Gamma},
    & \qquad t \in (0,T], \\
    \boldsymbol{W}(0) = \boldsymbol{W}_{0}.
    \end{dcases}
\end{equation}
\par
To advance \eqref{eq:algebraic} in time, we adopt an implicit--explicit
Runge--Kutta (IMEX--RK) time discretization, in which the linear diffusion part is treated implicitly, while the nonlinear reaction
term $\boldsymbol{G}(\boldsymbol{W})+\boldsymbol{\Gamma}$ is treated
explicitly. More precisely, letting $0=t^0<t^1<\dots<t^{N_T}=T$ with
time-step $\Delta t$, and denoting by $\boldsymbol{W}^m\approx
\boldsymbol{W}(t^m)$, we use an $s$-stage IMEX--RK scheme of the form
\[
\boldsymbol{W}^{m+1}
= \boldsymbol{W}^{m}
+ \Delta t \sum_{i=1}^s \tilde{b}_i
\bigl(\boldsymbol{G}(\boldsymbol{W}^{(i)}) + \boldsymbol{\Gamma}\bigr)
- \Delta t \sum_{i=1}^s b_i\,\mathbf{K}\boldsymbol{W}^{(i)},
\]
where the internal stages $\boldsymbol{W}^{(i)}$ are defined by a pair of
Butcher tableaux $(\tilde{A},\tilde{b})$ and $(A,b)$ corresponding to a
high-order optimized IMEX--RK method, as in
\cite{antonietti_optimized_2026} for diffusion--reaction problems.
\begin{remark}
The IMEX--RK discretization is introduced for the numerical experiments; a fully discrete stability and error analysis is beyond the scope of the present work. In all convergence studies, the time step is selected sufficiently small so that the temporal contribution is negligible with respect to the spatial discretization error.
\end{remark}
\section{Numerical results}
\label{sec:numerical_results}
In this section, we present numerical experiments aimed at validating the theoretical results and illustrating the applicability of the proposed framework to representative reaction--diffusion models. In Section~\ref{sec:tc1}, we consider a three-species Lotka--Volterra competition--diffusion system with a known travelling-wave solution and verify the expected spatial convergence rates. Section~\ref{sec:tc2} addresses a manufactured two-species system featuring both self- and cross-species cubic higher-order interactions, for which we investigate $h$-convergence in the dG, $L^2(\Omega)$, and energy norms, as well as the behaviour under $p$-refinement. Finally, in Section~\ref{sec:tc3}, we consider a heterogeneous coupled Fisher--Kolmogorov model with discontinuous diffusion coefficients to assess the robustness of the weighted interior penalty discretization in a qualitative setting. The implementation of the PolyDG method is carried out within the \texttt{lymph} library framework~\cite{antonietti_lymph_2025}.

\subsection{Test case 1: Three-species competition-diffusion system}
\label{sec:tc1}
In this first test case, we consider the three-species competition-diffusion system studied in \cite{chen_exact_2012}, for which the existence of an analytical travelling-wave solution is proved. The problem reads:
\begin{equation}
\label{eq:lotka_volterra_3}
    \begin{dcases}
        \frac{\partial w_1}{\partial t} = \Delta w_1 + w_1\left(1-w_1-\frac{8}{5}w_2 - 2 w_3\right), & \mathrm{in}\,\Omega\times(0,T], \\
        \frac{\partial w_2}{\partial t} = \Delta w_2 + w_2\left(1-\frac{19}{25}w_1-w_2 - \frac{1}{16} w_3\right), & \mathrm{in}\,\Omega\times(0,T], \\
        \frac{\partial w_3}{\partial t} = \Delta w_3 + w_3\left(1-\frac{18}{25}w_1-\frac{8}{5}w_2 - w_3\right), & \mathrm{in}\,\Omega\times(0,T],
    \end{dcases}
\end{equation}
and is equivalent to problem~\eqref{eq:general_strong} without cubic interactions ($\mathbb{C} = \boldsymbol{0}$), setting
\begin{equation*}
   \mathbf{L} = \mathbf{I}, \qquad \mathbf{Q}\boldsymbol{w} = \mathrm{diag}\left(-w_1-\frac{8}{5}w_2 - 2 w_3,\, -\frac{19}{25}w_1-w_2 - \frac{1}{16} w_3,\, -\frac{18}{25}w_1-\frac{8}{5}w_2 - w_3\right),
\end{equation*}
where $\mathbf{I}$ is the identity matrix. Considering suitable boundary and initial conditions, the analytical solution of this problem is
\begin{equation}
    w_1 = \frac{1}{2}\left(1+\tanh\left(\frac{x-\sigma t}{5}\right)\right), \quad
    w_2 = \frac{1}{4}\left(1-\tanh\left(\frac{x-\sigma t}{5}\right)\right)^2, \quad
    w_3 = \frac{4}{25}\left(1-\tanh^2\left(\frac{x-\sigma t}{5}\right)\right),
\end{equation}
with $\sigma = \frac{11}{10}$ as proved in \cite{chen_exact_2012}. We consider the rectangular spatial domain $\Omega = (-30,30) \times (0,10)$ and set the final time to $T = 10^{-3}$. Dirichlet boundary conditions are prescribed by evaluating the exact solution on $\partial\Omega$. The computational domain is discretized using a sequence of polygonal meshes generated with PolyMesher~\cite{talischi_polymesher_2012}, consisting of $N_{\mathrm{el}} = 50, 100, 200, 400, 800, 1600, 3200$ elements. We consider polynomial degrees $\ell = 1,2,3,4$. To investigate the spatial convergence of the proposed method, we fix the time step to $\Delta t = 10^{-5}$ and employ the $\mathrm{IMEX}(4,6)$-$\mathrm{LD}_{p_3}$ scheme introduced in~\cite{antonietti_optimized_2026}.
\par
\begin{figure}[t!]
    \begin{subfigure}[b]{0.49\textwidth}
        \resizebox{\textwidth}{!}{\begin{tikzpicture}

\begin{axis}[%
    width           = 3.875in,
    height          = 2.500in,
    at              = {(2.6in,1.099in)},
    scale only axis,
    xmode           = log,
    xmin            = 8.3589e-01,
    xmax            = 6.5218e+00,
    xminorticks     = true,
    xlabel          = {$h$},
    ylabel          = {$\|\boldsymbol{w}(T) -\boldsymbol{w}_h(T)\|_{e_1}$},
    xticklabel      = {\pgfmathparse{exp(\tick)}\pgfmathprintnumber{\pgfmathresult}},
    x tick label style ={/pgf/number format/.cd, 
                         fixed, fixed zerofill,
                         precision=2},
    ymode           = log,
    ymin            = 1e-8,
    ymax            = 2e-1,
    yminorticks     = true,
    axis background/.style  = {fill=white},
    title style     = {font=\bfseries},
    xmajorgrids,
    xminorgrids,
    ymajorgrids,
    yminorgrids,
    legend style    = {at={(0.96,0.35)},legend cell align=left, draw=white!15!black}]
              
\addplot [color=red, line width=2.0pt, mark=*, mark options=solid]
  table[row sep=crcr]{%
6.5218e+00      1.2250e-01\\
4.2648e+00      5.9257e-02\\
3.3142e+00      2.9509e-02\\
2.3517e+00      1.6383e-02\\
1.6129e+00      8.4374e-03\\
1.2026e+00      4.5646e-03\\
8.3589e-01      2.6461e-03\\
};
\addlegendentry{$\ell=1$}

\addplot [color=violet, line width=2.0pt, mark=*, mark options=solid]
  table[row sep=crcr]{
6.5218e+00      1.5400e-02\\
4.2648e+00      5.4431e-03\\
3.3142e+00      2.2077e-03\\
2.3517e+00      1.0019e-03\\
1.6129e+00      4.1966e-04\\
1.2026e+00      1.7444e-04\\
8.3589e-01      7.8179e-05\\
};
\addlegendentry{$\ell=2$}

\addplot [color=blue, line width=2.0pt, mark=*, mark options=solid]
  table[row sep=crcr]{
6.5218e+00      2.1000e-03\\
4.2648e+00      6.1136e-04\\
3.3142e+00      1.4338e-04\\
2.3517e+00      5.5287e-05\\
1.6129e+00      1.4464e-05\\
1.2026e+00      4.5275e-06\\
8.3589e-01      1.3426e-06\\
};
\addlegendentry{$\ell=3$}

\addplot [color=teal, line width=2.0pt, mark=*, mark options=solid]
  table[row sep=crcr]{
6.5218e+00      2.6187e-04\\
4.2648e+00      5.4496e-05\\
3.3142e+00      1.1076e-05\\
2.3517e+00      3.5443e-06\\
1.6129e+00      6.1924e-07\\
1.2026e+00      1.4350e-07\\
8.3589e-01      3.1305e-08\\
};
\addlegendentry{$\ell=4$}

\node[right, align=left, text=black, font=\footnotesize]
at (axis cs:2.005,0.00525) {$1$};

\addplot [color=black, line width=1.5pt]
  table[row sep=crcr]{%
2.00   6.00e-03 \\
1.50   4.50e-03 \\
2.00   4.50e-03 \\
2.00   6.00e-03 \\
};

\node[right, align=left, text=black, font=\footnotesize]
at (axis cs:2.005,2.8e-04) {$2$};

\addplot [color=black, line width=1.5pt]
  table[row sep=crcr]{%
2.00   4.00e-04 \\
1.50   2.25e-04 \\
2.00   2.25e-04 \\
2.00   4.00e-04 \\
};

\node[right, align=left, text=black, font=\footnotesize]
at (axis cs:2.005,1.00e-5) {$3$};

\addplot [color=black, line width=1.5pt]
  table[row sep=crcr]{%
2.00   1.60e-05 \\
1.50   6.72e-06 \\
2.00   6.72e-06 \\
2.00   1.60e-05 \\
};

\node[right, align=left, text=black, font=\footnotesize]
at (axis cs:2.005,3.5e-7) {$4$};

\addplot [color=black, line width=1.5pt]
  table[row sep=crcr]{%
2.00   6.00e-07\\
1.50   1.89e-07\\
2.00   1.89e-07\\
2.00   6.00e-07\\
};

\end{axis}
\end{tikzpicture}
}
        \caption{Computed errors in energy norm w.r.t.~the mesh size~$h$.}
        \label{fig:LV_hconvergence_e1}
    \end{subfigure}
    \begin{subfigure}[b]{0.49\textwidth}
        \resizebox{\textwidth}{!}{\begin{tikzpicture}

\begin{axis}[%
    width           = 3.875in,
    height          = 2.500in,
    at              = {(2.6in,1.099in)},
    scale only axis,
    xmode           = log,
    xmin            = 8.3589e-01,
    xmax            = 6.5218e+00,
    xminorticks     = true,
    xlabel          = {$h$},
    ylabel          = {$\|\boldsymbol{w}(T) -\boldsymbol{w}_h(T)\|_{\mathbf{L}^2(\Omega)}$},
    xticklabel      = {\pgfmathparse{exp(\tick)}\pgfmathprintnumber{\pgfmathresult}},
    x tick label style ={/pgf/number format/.cd, 
                         fixed, fixed zerofill,
                         precision=2},
    ymode           = log,
    ymin            = 1e-8,
    ymax            = 2e-1,
    yminorticks     = true,
    axis background/.style  = {fill=white},
    title style     = {font=\bfseries},
    xmajorgrids,
    xminorgrids,
    ymajorgrids,
    yminorgrids,
    legend style    = {at={(0.96,0.35)},legend cell align=left, draw=white!15!black}]
              
\addplot [color=red, line width=2.0pt, mark=square*, mark options=solid]
  table[row sep=crcr]{%
6.5218e+00      1.2160e-01\\
4.2648e+00      5.8402e-02\\
3.3142e+00      2.8720e-02\\
2.3517e+00      1.5613e-02\\
1.6129e+00      7.7282e-03\\
1.2026e+00      3.8843e-03\\
8.3589e-01      2.0046e-03\\
};
\addlegendentry{$\ell=1$}

\addplot [color=violet, line width=2.0pt, mark=square*, mark options=solid]
  table[row sep=crcr]{
6.5218e+00      1.4700e-02\\
4.2648e+00      4.9097e-03\\
3.3142e+00      1.8141e-03\\
2.3517e+00      7.9861e-04\\
1.6129e+00      3.0212e-04\\
1.2026e+00      1.2043e-04\\
8.3589e-01      4.6690e-05\\
};
\addlegendentry{$\ell=2$}

\addplot [color=blue, line width=2.0pt, mark=square*, mark options=solid]
  table[row sep=crcr]{
6.5218e+00      1.9000e-03\\
4.2648e+00      5.2616e-04\\
3.3142e+00      1.2189e-04\\
2.3517e+00      4.3656e-05\\
1.6129e+00      1.0859e-05\\
1.2026e+00      3.2162e-06\\
8.3589e-01      8.5627e-07\\
};
\addlegendentry{$\ell=3$}

\addplot [color=teal, line width=2.0pt, mark=square*, mark options=solid]
  table[row sep=crcr]{
6.5218e+00      2.4364e-04\\
4.2648e+00      4.5992e-05\\
3.3142e+00      8.9665e-06\\
2.3517e+00      2.5923e-06\\
1.6129e+00      4.3572e-07\\
1.2026e+00      9.1900e-08\\
8.3589e-01      1.8816e-08\\
};
\addlegendentry{$\ell=4$}

\node[right, align=left, text=black, font=\footnotesize]
at (axis cs:2.005,0.0045) {$2$};

\addplot [color=black, line width=1.5pt]
  table[row sep=crcr]{%
2.00   6.000e-03 \\
1.50   3.375e-03 \\
2.00   3.375e-03 \\
2.00   6.000e-03 \\
};

\node[right, align=left, text=black, font=\footnotesize]
at (axis cs:2.005,1.4e-04) {$3$};

\addplot [color=black, line width=1.5pt]
  table[row sep=crcr]{%
2.00   2.00e-04 \\
1.50   8.40e-05 \\
2.00   8.40e-05 \\
2.00   2.00e-04 \\
};

\node[right, align=left, text=black, font=\footnotesize]
at (axis cs:2.005,7.0e-6) {$4$};

\addplot [color=black, line width=1.5pt]
  table[row sep=crcr]{%
2.00   1.20e-05\\
1.50   3.78e-06\\
2.00   3.78e-06\\
2.00   1.20e-05\\
};

\node[right, align=left, text=black, font=\footnotesize]
at (axis cs:2.005,3.0e-7) {$5$};

\addplot [color=black, line width=1.5pt]
  table[row sep=crcr]{%
2.00   6.000e-07\\
1.50   1.423e-07\\
2.00   1.423e-07\\
2.00   6.000e-07\\
};

\end{axis}
\end{tikzpicture}
}
          \caption{Computed errors in $\mathbf{L}^2(\Omega)$-norm
          w.r.t.~the mesh size $h$.}
        \label{fig:LV_hconvergence_L2}
    \end{subfigure}
    \caption{Test case 1: Computed errors and convergence rates w.r.t.~the mesh size~$h$.}
    \label{fig:LV_hconvergence}
\end{figure}
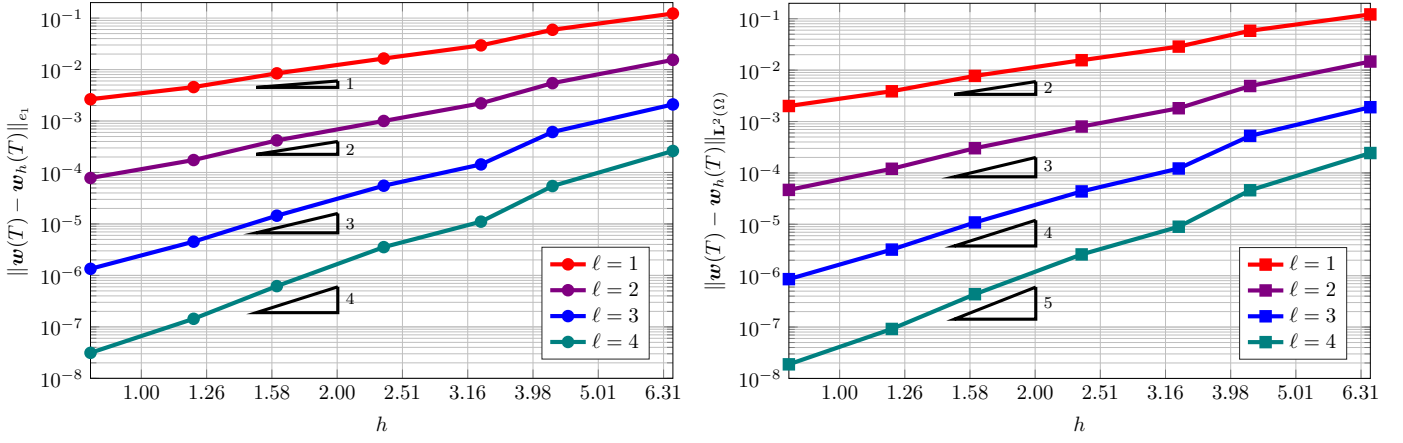
Figures~\ref{fig:LV_hconvergence_e1} and~\ref{fig:LV_hconvergence_L2} show the computed errors in the energy norm and in the $\mathbf{L}^2(\Omega)$-norm, respectively, together with the corresponding convergence rates. In agreement with the theoretical a priori estimates, the observed rates match the expected orders $O(h^\ell)$ in the energy norm. Although not covered by the present analysis the optimal order $O(h^{\ell+1})$ in the $\mathbf{L}^2(\Omega)$-norm is observed. The observed rates are already consistent with the asymptotic regime over the considered range of mesh sizes.
\par
\begin{figure}
    \includegraphics[width=\textwidth]{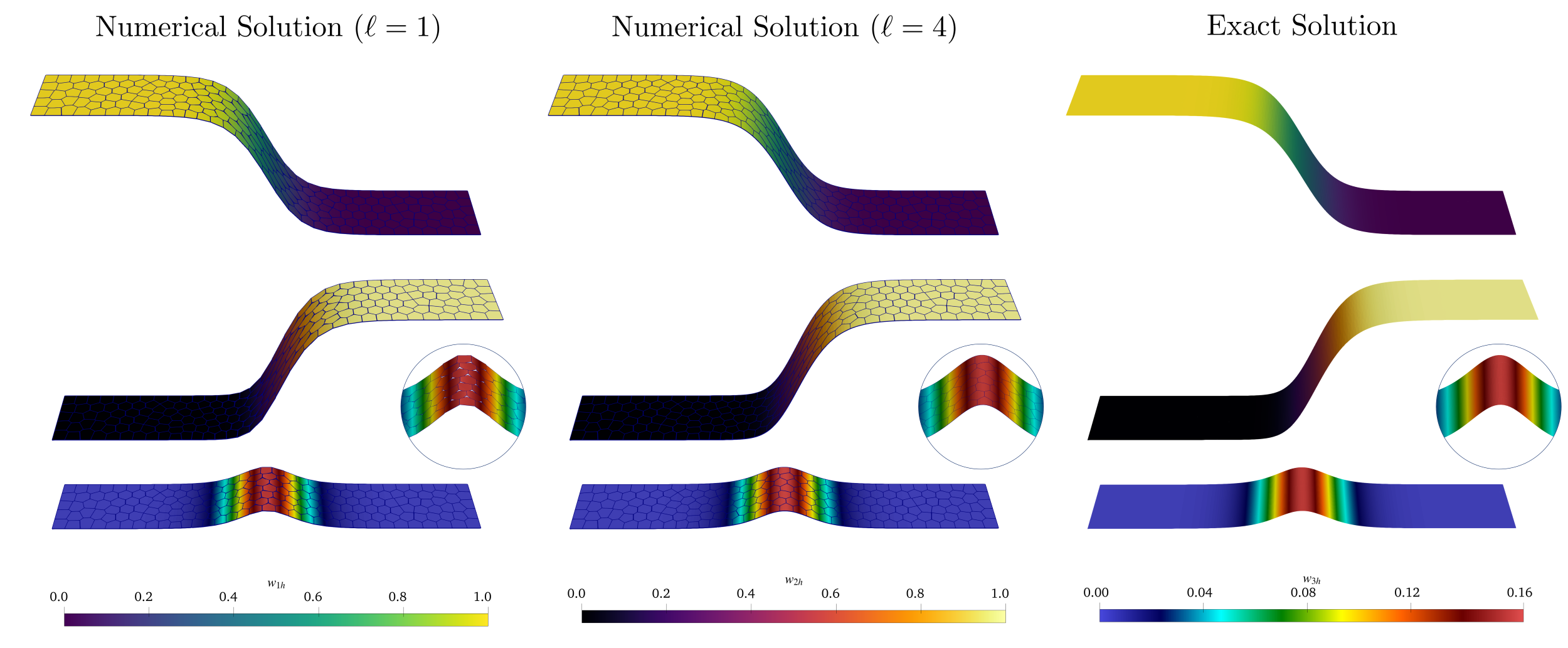}
    \caption{Test case 1: Qualitative representation of the numerical solutions for $\ell=1$ (first column) and $\ell=4$ (second column) together with the exact solution (third column) at time $T=10^{-3}$.}
    \label{fig:LV_solution}
\end{figure}
Figure~\ref{fig:LV_solution} provides a qualitative comparison between the numerical solutions obtained with polynomial degrees $\ell=1$ and $\ell=4$ on a mesh consisting of $200$ polygonal elements and the exact travelling-wave solution, for each of the three species $w_1$, $w_2$, and $w_3$. Although the low-order approximation with $\ell=1$ correctly captures the overall shape and propagation of the travelling wave, a noticeable loss of accuracy is observed in the transition region. This is particularly evident for $w_3$, where the enlarged view highlights a less sharp profile and small numerical oscillations. In contrast, the approximation obtained with $\ell=4$ is visually indistinguishable from the exact solution, including in the magnified region around the wave front. This qualitative agreement is consistent with the higher accuracy and convergence rates reported in Figures~\ref{fig:LV_hconvergence_e1} and~\ref{fig:LV_hconvergence_L2}.
\subsection{Test case 2: Two-species system with higher-order interaction terms}
\label{sec:tc2}
In this second test case, we consider the two-species reaction-diffusion system with higher-order terms. The problem reads:
\begin{equation}
\label{eq:tc2_equations}
    \begin{dcases}
        \frac{\partial w_1}{\partial t} = \Delta w_1 + 2 w_1^2\left(1-w_1\right), & \mathrm{in}\,\Omega\times(0,T], \\
        \frac{\partial w_2}{\partial t} = \Delta w_2 + w_2\left(w_1+w_2\right)-2w_1^2w_2, & \mathrm{in}\,\Omega\times(0,T], \\
    \end{dcases}
\end{equation}
and is equivalent to problem~\eqref{eq:general_strong}, setting
\begin{equation*}
   \mathbf{L} = \mathbf{0}, \qquad \mathbf{Q}\boldsymbol{w} = \begin{pmatrix}
        2w_1 & 0 \\
        0 & w_1+w_2
    \end{pmatrix},
\qquad
(\mathbb{C}\boldsymbol{w})\boldsymbol{w}
= \begin{pmatrix}
-2w_1^2 & 0 \\
0 & -2w_1^2
\end{pmatrix}.
\end{equation*}
where $\mathbf{0}$ is the null matrix. Considering suitable boundary and initial conditions, the analytical solution of this problem is $w_1 = w_2 = 0.5 \left(1-\tanh\left(0.5(x - t)\right)\right)$. The model in \eqref{eq:tc2_equations} is a manufactured example with analytical solution to study the properties of convergence of the proposed numerical method.
We employ the same computational setting as in the previous test case, considering $\Omega=(-30,30)\times(0,10)$, $T=10^{-3}$, and the same family of computational meshes. Dirichlet boundary conditions are derived from the exact solution. The spatial convergence is investigated for $\ell=1,2,3,4$, using $\Delta t=10^{-5}$ and the $\mathrm{IMEX}(4,6)$-$\mathrm{LD}_{p_3}$ scheme of~\cite{antonietti_optimized_2026}.
\par
Figures~\ref{fig:HOI_hconvergence_dG} and~\ref{fig:HOI_hconvergence_L2} display the computed errors in the dG and $L^2(\Omega)$ norms, respectively. The observed convergence rates agree with the expected orders, namely $O(h^\ell)$ in the dG norm and $O(h^{\ell+1})$ in the $L^2(\Omega)$ norm. Moreover, consistently with the a priori estimate established in Theorem~\ref{thm:apriori}, Figure~\ref{fig:HOI_hconvergence_e1} shows that the energy functional converges with order $O(h^\ell)$.
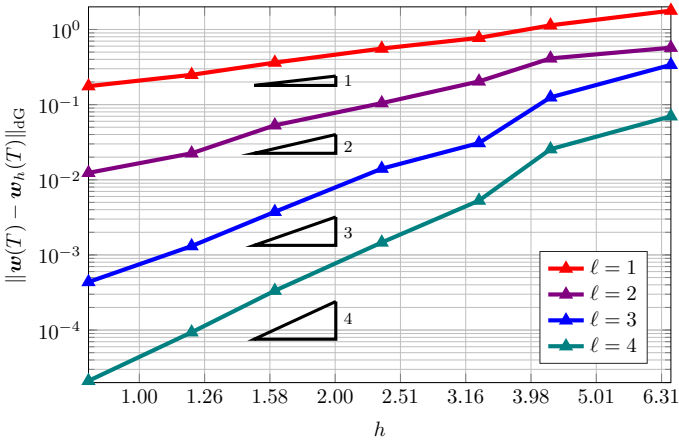
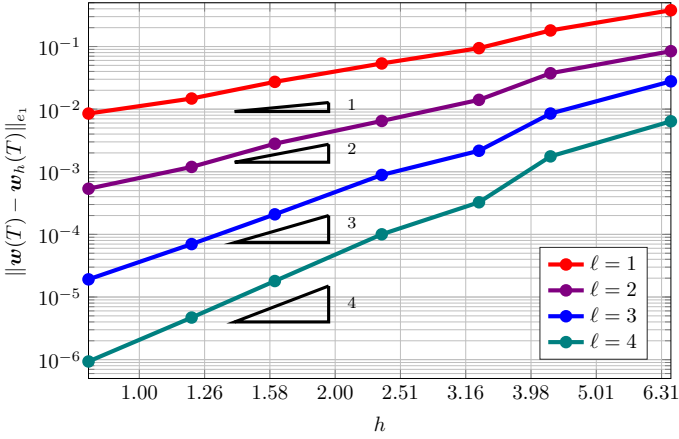
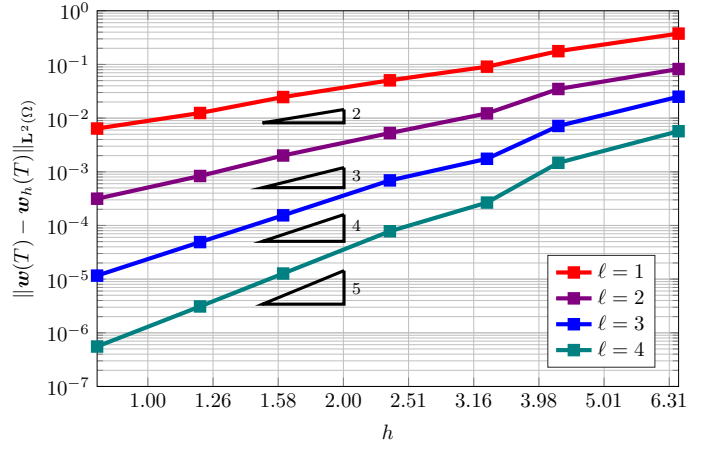
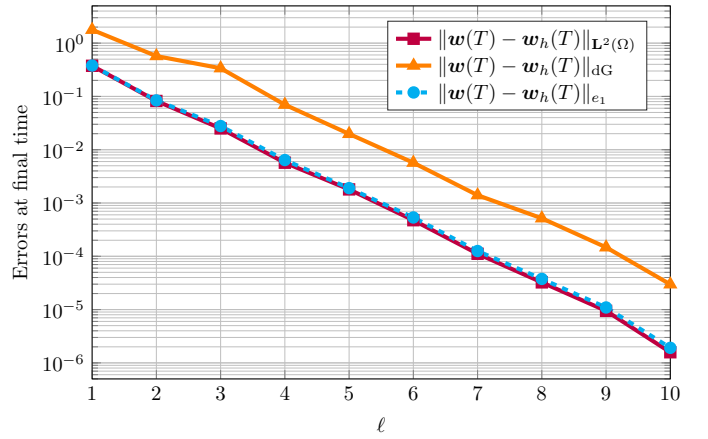
\begin{figure}[t!]
    \begin{subfigure}[b]{0.49\textwidth}
        \resizebox{\textwidth}{!}{\begin{tikzpicture}

\begin{axis}[%
    width           = 3.875in,
    height          = 2.500in,
    at              = {(2.6in,1.099in)},
    scale only axis,
    xmode           = log,
    xmin            = 0.83589,
    xmax            = 6.5218,
    xminorticks     = true,
    xlabel          = {$h$},
    ylabel          = {$\|\boldsymbol{w}(T) -\boldsymbol{w}_h(T)\|_{\mathrm{dG}}$},
    xticklabel      = {\pgfmathparse{exp(\tick)}\pgfmathprintnumber{\pgfmathresult}},
    x tick label style ={/pgf/number format/.cd, 
                         fixed, fixed zerofill,
                         precision=2},
    ymode           = log,
    ymin            = 2e-5,
    ymax            = 2.00,
    yminorticks     = true,
    axis background/.style  = {fill=white},
    title style     = {font=\bfseries},
    xmajorgrids,
    xminorgrids,
    ymajorgrids,
    yminorgrids,
    legend style    = {at={(0.96,0.35)},legend cell align=left, draw=white!15!black}]
              
\addplot [color=red, line width=2.0pt, mark=triangle*, mark options=solid]
  table[row sep=crcr]{%
   6.5218e+00   1.7794e+00 \\
   4.2648e+00   1.1357e+00 \\
   3.3142e+00   7.7360e-01 \\
   2.3517e+00   5.5685e-01 \\
   1.6129e+00   3.6242e-01 \\
   1.2026e+00   2.4912e-01 \\
   8.3589e-01   1.7600e-01 \\
};
\addlegendentry{$\ell=1$}

\addplot [color=violet, line width=2.0pt, mark=triangle*, mark options=solid]
  table[row sep=crcr]{
   6.5218e+00   5.7181e-01 \\
   4.2648e+00   4.1200e-01 \\
   3.3142e+00   2.0360e-01 \\
   2.3517e+00   1.0467e-01 \\
   1.6129e+00   5.3045e-02 \\
   1.2026e+00   2.2498e-02 \\
   8.3589e-01   1.2363e-02 \\
};
\addlegendentry{$\ell=2$}

\addplot [color=blue, line width=2.0pt, mark=triangle*, mark options=solid]
  table[row sep=crcr]{
   6.5218e+00   3.3797e-01 \\
   4.2648e+00   1.2498e-01 \\
   3.3142e+00   3.0801e-02 \\
   2.3517e+00   1.4056e-02 \\
   1.6129e+00   3.7433e-03 \\
   1.2026e+00   1.3102e-03 \\
   8.3589e-01   4.3621e-04 \\
};
\addlegendentry{$\ell=3$}

\addplot [color=teal, line width=2.0pt, mark=triangle*, mark options=solid]
  table[row sep=crcr]{
   6.5218e+00   6.9913e-02 \\
   4.2648e+00   2.5516e-02 \\
   3.3142e+00   5.2824e-03 \\
   2.3517e+00   1.4638e-03 \\
   1.6129e+00   3.3388e-04 \\
   1.2026e+00   9.3394e-05 \\
   8.3589e-01   2.1090e-05 \\
};
\addlegendentry{$\ell=4$}

\node[right, align=left, text=black, font=\footnotesize]
at (axis cs:2.005,0.210) {$1$};

\addplot [color=black, line width=1.5pt]
  table[row sep=crcr]{%
2.00   2.40e-01 \\
1.50   1.80e-01 \\
2.00   1.80e-01 \\
2.00   2.40e-01 \\
};

\node[right, align=left, text=black, font=\footnotesize]
at (axis cs:2.005,2.8e-02) {$2$};

\addplot [color=black, line width=1.5pt]
  table[row sep=crcr]{%
2.00   4.00e-02 \\
1.50   2.25e-02 \\
2.00   2.25e-02 \\
2.00   4.00e-02 \\
};

\node[right, align=left, text=black, font=\footnotesize]
at (axis cs:2.005,2.00e-3) {$3$};

\addplot [color=black, line width=1.5pt]
  table[row sep=crcr]{%
2.00   3.20e-03 \\
1.50   1.34e-03 \\
2.00   1.34e-03 \\
2.00   3.20e-03 \\
};

\node[right, align=left, text=black, font=\footnotesize]
at (axis cs:2.005,1.4e-4) {$4$};

\addplot [color=black, line width=1.5pt]
  table[row sep=crcr]{%
2.00   2.40e-04\\
1.50   7.56e-05\\
2.00   7.56e-05\\
2.00   2.40e-04\\
};

\end{axis}
\end{tikzpicture}
}
        \caption{Computed errors in dG norm w.r.t.~the mesh size~$h$.}
        \label{fig:HOI_hconvergence_dG}
    \end{subfigure}
    \begin{subfigure}[b]{0.49\textwidth}
        \resizebox{\textwidth}{!}{\begin{tikzpicture}

\begin{axis}[%
    width           = 3.875in,
    height          = 2.500in,
    at              = {(2.6in,1.099in)},
    scale only axis,
    xmode           = log,
    xmin            = 0.83589,
    xmax            = 6.5218,
    xminorticks     = true,
    xlabel          = {$h$},
    ylabel          = {$\|\boldsymbol{w}(T) -\boldsymbol{w}_h(T)\|_{\mathbf{L}^2(\Omega)}$},
    xticklabel      = {\pgfmathparse{exp(\tick)}\pgfmathprintnumber{\pgfmathresult}},
    x tick label style ={/pgf/number format/.cd, 
                         fixed, fixed zerofill,
                         precision=2},
    ymode           = log,
    ymin            = 1e-7,
    ymax            = 1e-0,
    yminorticks     = true,
    axis background/.style  = {fill=white},
    title style     = {font=\bfseries},
    xmajorgrids,
    xminorgrids,
    ymajorgrids,
    yminorgrids,
    legend style    = {at={(0.96,0.35)},legend cell align=left, draw=white!15!black}]
              
\addplot [color=red, line width=2.0pt, mark=square*, mark options=solid]
  table[row sep=crcr]{%
   6.5218e+00   3.7570e-01 \\
   4.2648e+00   1.7691e-01 \\
   3.3142e+00   9.0917e-02 \\
   2.3517e+00   5.0612e-02 \\
   1.6129e+00   2.4696e-02 \\ 
   1.2026e+00   1.2461e-02 \\
   8.3589e-01   6.4039e-03 \\
};
\addlegendentry{$\ell=1$}

\addplot [color=violet, line width=2.0pt, mark=square*, mark options=solid]
  table[row sep=crcr]{
   6.5218e+00   8.2190e-02 \\
   4.2648e+00   3.4908e-02 \\
   3.3142e+00   1.2246e-02 \\
   2.3517e+00   5.2711e-03 \\
   1.6129e+00   2.0120e-03 \\
   1.2026e+00   8.3526e-04 \\
   8.3589e-01   3.1506e-04 \\
};
\addlegendentry{$\ell=2$}

\addplot [color=blue, line width=2.0pt, mark=square*, mark options=solid]
  table[row sep=crcr]{
   6.5218e+00   2.5063e-02 \\
   4.2648e+00   7.1427e-03 \\
   3.3142e+00   1.7477e-03 \\
   2.3517e+00   6.9117e-04 \\
   1.6129e+00   1.5438e-04 \\
   1.2026e+00   4.9007e-05 \\
   8.3589e-01   1.1621e-05 \\
};
\addlegendentry{$\ell=3$}

\addplot [color=teal, line width=2.0pt, mark=square*, mark options=solid]
  table[row sep=crcr]{
   6.5218e+00   5.7045e-03 \\
   4.2648e+00   1.4710e-03 \\
   3.3142e+00   2.6693e-04 \\
   2.3517e+00   7.8002e-05 \\
   1.6129e+00   1.2716e-05 \\
   1.2026e+00   3.0966e-06 \\
   8.3589e-01   5.5439e-07 \\
};
\addlegendentry{$\ell=4$}

\node[right, align=left, text=black, font=\footnotesize]
at (axis cs:2.005,1.20e-02) {$2$};

\addplot [color=black, line width=1.5pt]
  table[row sep=crcr]{%
2.00   1.45e-02 \\
1.50   8.16e-03 \\
2.00   8.16e-03 \\
2.00   1.45e-02 \\
};

\node[right, align=left, text=black, font=\footnotesize]
at (axis cs:2.005,8.50e-04) {$3$};

\addplot [color=black, line width=1.5pt]
  table[row sep=crcr]{%
2.00   1.20e-03 \\
1.50   5.06e-04 \\
2.00   5.06e-04 \\
2.00   1.20e-03 \\
};

\node[right, align=left, text=black, font=\footnotesize]
at (axis cs:2.005,1.00e-04) {$4$};

\addplot [color=black, line width=1.5pt]
  table[row sep=crcr]{%
2.00   1.60e-04 \\
1.50   5.06e-05 \\
2.00   5.06e-05 \\
2.00   1.60e-04 \\
};

\node[right, align=left, text=black, font=\footnotesize]
at (axis cs:2.005,6.60e-06) {$5$};

\addplot [color=black, line width=1.5pt]
  table[row sep=crcr]{%
2.00   1.440e-05 \\
1.50   3.416e-06 \\
2.00   3.416e-06 \\
2.00   1.440e-05 \\
};

\end{axis}
\end{tikzpicture}
}
          \caption{Computed errors in $\mathbf{L}^2(\Omega)$-norm
          w.r.t.~the mesh size $h$.}
        \label{fig:HOI_hconvergence_L2}
    \end{subfigure}
    \begin{subfigure}[b]{0.49\textwidth}
        \resizebox{\textwidth}{!}{\begin{tikzpicture}

\begin{axis}[%
    width           = 3.875in,
    height          = 2.500in,
    at              = {(2.6in,1.099in)},
    scale only axis,
    xmode           = log,
    xmin            = 0.83589,
    xmax            = 6.5218,
    xminorticks     = true,
    xlabel          = {$h$},
    ylabel          = {$\|\boldsymbol{w}(T) -\boldsymbol{w}_h(T)\|_{e_1}$},
    xticklabel      = {\pgfmathparse{exp(\tick)}\pgfmathprintnumber{\pgfmathresult}},
    x tick label style ={/pgf/number format/.cd, 
                         fixed, fixed zerofill,
                         precision=2},
    ymode           = log,
    ymin            = 5e-7,
    ymax            = 5e-1,
    yminorticks     = true,
    axis background/.style  = {fill=white},
    title style     = {font=\bfseries},
    xmajorgrids,
    xminorgrids,
    ymajorgrids,
    yminorgrids,
    legend style    = {at={(0.96,0.35)},legend cell align=left, draw=white!15!black}]
              
\addplot [color=red, line width=2.0pt, mark=*, mark options=solid]
  table[row sep=crcr]{%
   6.5218e+00   3.7990e-01 \\
   4.2648e+00   1.8054e-01 \\
   3.3142e+00   9.4182e-02 \\
   2.3517e+00   5.3627e-02 \\
   1.6129e+00   2.7268e-02 \\
   1.2026e+00   1.4776e-02 \\
   8.3589e-01   8.5049e-03 \\ 
};
\addlegendentry{$\ell=1$}

\addplot [color=violet, line width=2.0pt, mark=*, mark options=solid]
  table[row sep=crcr]{
   6.5218e+00   8.4250e-02 \\
   4.2648e+00   3.7426e-02 \\
   3.3142e+00   1.4073e-02 \\
   2.3517e+00   6.4752e-03 \\
   1.6129e+00   2.8044e-03 \\
   1.2026e+00   1.1948e-03 \\
   8.3589e-01   5.3598e-04 \\
};
\addlegendentry{$\ell=2$}

\addplot [color=blue, line width=2.0pt, mark=*, mark options=solid]
  table[row sep=crcr]{
   6.5218e+00   2.7792e-02 \\
   4.2648e+00   8.5062e-03 \\
   3.3142e+00   2.1643e-03 \\
   2.3517e+00   8.8918e-04 \\
   1.6129e+00   2.0899e-04 \\
   1.2026e+00   7.0009e-05 \\
   8.3589e-01   1.9148e-05 \\
};
\addlegendentry{$\ell=3$}

\addplot [color=teal, line width=2.0pt, mark=*, mark options=solid]
  table[row sep=crcr]{
   6.5218e+00   6.3985e-03 \\
   4.2648e+00   1.7608e-03 \\
   3.3142e+00   3.2552e-04 \\
   2.3517e+00   1.0029e-04 \\
   1.6129e+00   1.7996e-05 \\
   1.2026e+00   4.6980e-06 \\
   8.3589e-01   9.3449e-07 \\
};
\addlegendentry{$\ell=4$}

\node[right, align=left, text=black, font=\footnotesize]
at (axis cs:2.03,1.1875e-02) {$1$};

\addplot [color=black, line width=1.5pt]
  table[row sep=crcr]{%
1.95   1.275e-02 \\
1.40   9.150e-03 \\
1.95   9.150e-03 \\
1.95   1.275e-02 \\
};

\node[right, align=left, text=black, font=\footnotesize]
at (axis cs:2.03,2.30e-03) {$2$};

\addplot [color=black, line width=1.5pt]
  table[row sep=crcr]{%
1.95   2.75e-03 \\
1.40   1.42e-03 \\
1.95   1.42e-03 \\
1.95   2.75e-03 \\
};

\node[right, align=left, text=black, font=\footnotesize]
at (axis cs:2.03,1.55e-04) {$3$};

\addplot [color=black, line width=1.5pt]
  table[row sep=crcr]{%
1.95   2.00e-04 \\
1.40   7.40e-05 \\
1.95   7.40e-05 \\
1.95   2.00e-04 \\
};

\node[right, align=left, text=black, font=\footnotesize]
at (axis cs:2.03,8.00e-06) {$4$};

\addplot [color=black, line width=1.5pt]
  table[row sep=crcr]{%
1.95   1.50e-05 \\
1.40   4.00e-06 \\
1.95   4.00e-06 \\
1.95   1.50e-05 \\
};

\end{axis}
\end{tikzpicture}
}
        \caption{Computed errors in energy functional w.r.t.~the mesh size~$h$.}
        \label{fig:HOI_hconvergence_e1}
    \end{subfigure}    
    \begin{subfigure}[b]{0.49\textwidth}
        \resizebox{\textwidth}{!}{\begin{tikzpicture}

\begin{axis}[%
    width           = 3.875in,
    height          = 2.500in,
    at              = {(2.6in,1.099in)},
    scale only axis,
    xmin            = 1,
    xmax            = 10,
    xminorticks     = true,
    xlabel          = {$\ell$},
    ylabel          = {Errors at final time},
    ymode           = log,
    ymin            = 5e-7,
    ymax            = 5e-0,
    yminorticks     = true,
    axis background/.style  = {fill=white},
    title style     = {font=\bfseries},
    xmajorgrids,
    xminorgrids,
    ymajorgrids,
    yminorgrids,
    legend style    = {at={(0.96,0.96)},legend cell align=left, draw=white!15!black}]

\addplot [color=purple, line width=2.0pt, mark=square*, mark options=solid]
  table[row sep=crcr]{%
    1       0.37570000 \\
    2       0.08220000 \\
    3       0.02510000 \\
    4       0.00570000 \\
    5       0.00180000 \\
    6       4.7543e-04 \\
    7       1.1152e-04 \\
    8       3.2413e-05 \\
    9       9.4669e-06 \\
    10      1.5817e-06 \\
};
\addlegendentry{$\|\boldsymbol{w}(T) -\boldsymbol{w}_h(T)\|_{\mathbf{L}^2(\Omega)}$}

\addplot [color=orange, line width=2.0pt, mark=triangle*, mark options=solid]
  table[row sep=crcr]{%
    1       1.77940000 \\
    2       0.57180000 \\
    3       0.33800000 \\
    4       0.06990000 \\
    5       0.01980000 \\
    6       0.00570000 \\
    7       0.00140000 \\
    8       5.1642e-04 \\
    9       1.4731e-04 \\
    10      2.9770e-05 \\
};
\addlegendentry{$\|\boldsymbol{w}(T) -\boldsymbol{w}_h(T)\|_{\mathrm{dG}}$}
              
\addplot [color=cyan, line width=2.0pt, mark=*, mark options=solid, dashed]
  table[row sep=crcr]{%
    1       0.37990000 \\
    2       0.08420000 \\
    3       0.02780000 \\
    4       0.00640000 \\
    5       0.00190000 \\
    6       5.3556e-04 \\
    7       1.2555e-04 \\
    8       3.7658e-05 \\
    9       1.0955e-05 \\
    10      1.9152e-06 \\
};
\addlegendentry{$\|\boldsymbol{w}(T) -\boldsymbol{w}_h(T)\|_{e_1}$}

\end{axis}
\end{tikzpicture}}
        \caption{Computed errors w.r.t.~the polynomial degree~$\ell$.}
        \label{fig:HOI_pconvergence}
    \end{subfigure}
    \caption{Test case 2: Computed errors and convergence rates w.r.t.~the mesh size~$h$ and polynomial degree $\ell$.}
    \label{fig:HOI_hconvergence}
\end{figure}
Finally, we investigate the convergence with respect to the polynomial degree by considering $\ell=1,\ldots,10$ on a fixed polygonal mesh consisting of $50$ elements. The corresponding results are reported in Figure~\ref{fig:HOI_pconvergence}. Although $p$-convergence is not covered by the present theoretical analysis, the errors decay rapidly as the polynomial degree increases, consistently with the smoothness of the analytical solution. This result highlights the effectiveness of high-order local polynomial approximations for the smooth travelling-wave solution considered in this test.
\par
This test simultaneously assesses the approximation properties of the proposed scheme in the presence of both self-interaction and cross-species higher-order reaction terms. The agreement between the observed and predicted $h$-convergence rates confirms that the cubic interaction does not compromise the optimal spatial accuracy of the method.
\subsection{Test case 3: A heterogeneous double Fisher--Kolmogorov system}
\label{sec:tc3}
As a final test case, we consider a system of two coupled Fisher--Kolmogorov equations:
\begin{equation}
\label{eq:doublefkpp}
    \begin{dcases}
        \dfrac{\partial w_1}{\partial t}
        =
        \nabla\cdot\left(\mu\nabla w_1\right)
        +
        w_1\left(1-w_1\right),
        & \mathrm{in}\,\Omega\times(0,T],
        \\
        \dfrac{\partial w_2}{\partial t}
        =
        \nabla\cdot\left(\mu\nabla w_2\right)
        +
        \left(1+w_1\right)w_2\left(1-w_2\right),
        & \mathrm{in}\,\Omega\times(0,T].
    \end{dcases}
\end{equation}
This system has been used, for instance, to model the spreading of \textit{prion-like} proteins in Alzheimer's disease; see~\cite{vazquez-palomo_computational_2026}.

We consider the square domain $\Omega=(-1,1)^2$ and the final time $T=36$. Homogeneous Neumann boundary conditions are imposed on $\partial\Omega$. The computational domain is discretized by a polygonal mesh consisting of $2500$ elements, generated with \textsc{PolyMesher}~\cite{talischi_polymesher_2012}. The spatial discretization employs polynomial degree $\ell=4$, while time integration is performed with the fourth-order $\mathrm{IMEX}(4,6)$-$\mathrm{LD}_{p_3}$ scheme introduced in~\cite{antonietti_optimized_2026}, using the time step $\Delta t=10^{-1}$.

To assess the robustness of the SWIP method in the presence of heterogeneous diffusion, we prescribe the discontinuous diffusion coefficient
\begin{equation}
\label{eq:doublefkpp_diffusion}
    \mu(x,y)
    =
    \begin{dcases}
        5\times10^{-3}, & x<0, \\
        5\times10^{-4}, & x\geq 0.
    \end{dcases}
\end{equation}
Finally, to investigate the effect of the first population $w_1$ on the growth of the second population $w_2$, we prescribe the initial conditions
\begin{equation}
\label{eq:doublefkpp_initial_data}
\begin{aligned}
    w_{1,0}(x,y)
    &=
    0.2\exp\left(
        -100\left((x-0.5)^2+(y+0.5)^2\right)
    \right)
    +
    0.2\exp\left(
        -100\left((x+0.5)^2+(y+0.5)^2\right)
    \right),
    \\
    w_{2,0}(x,y)
    &=
    0.5\exp\left(
        -100\left(x^2+y^2\right)
    \right),
\end{aligned}
\end{equation}
reported in Figure \ref{fig:DFKPP_solution} (first column).
\begin{figure}
    \includegraphics[width=\textwidth]{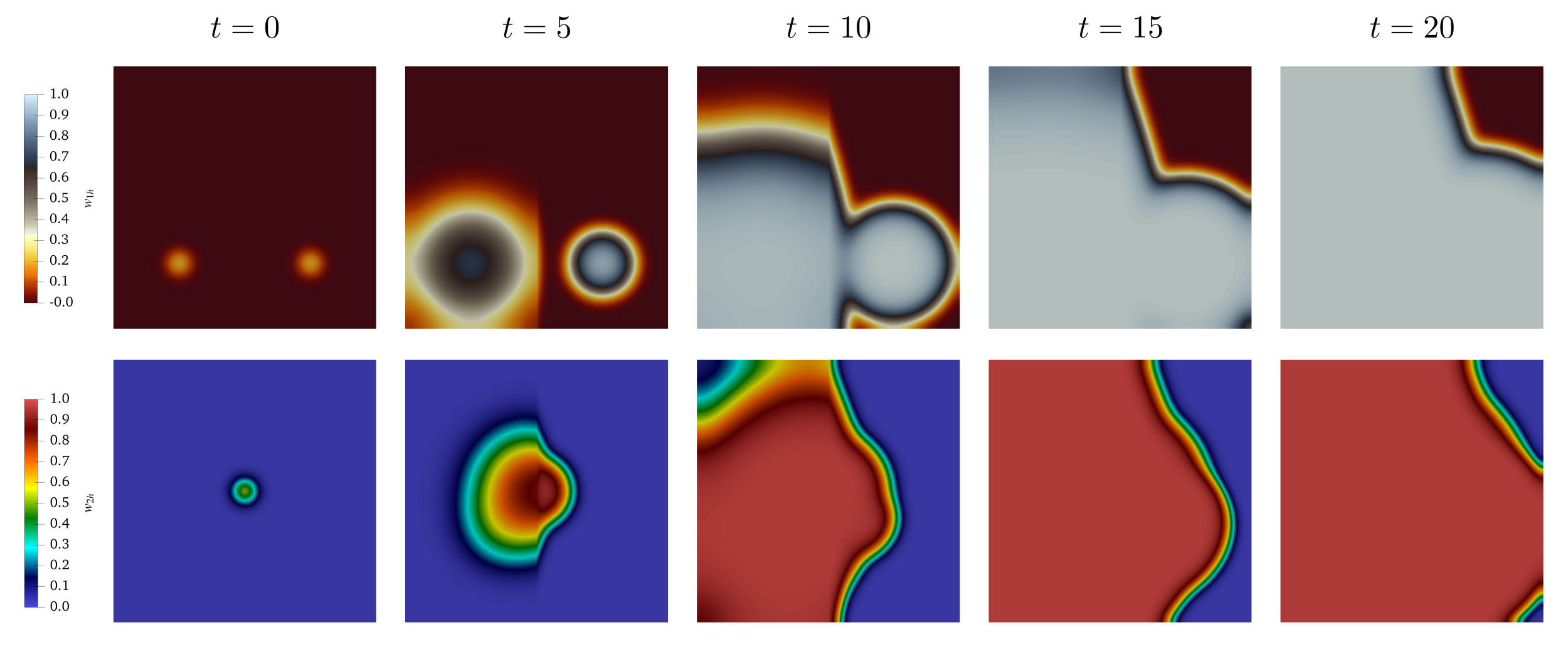}
    \caption{Test case 3: Qualitative representation of the numerical solutions for $w_{1h}$ (first row) and $w_{2h}$ (second row) for a sequence of time snapshots $t=0,5,10,15,20$.}
    \label{fig:DFKPP_solution}
\end{figure}
\par
Figure~\ref{fig:DFKPP_solution} illustrates the evolution of the two populations at $t=0,5,10,15$, and $20$. Initially, $w_1$ consists of two localized concentrations in the lower part of the domain, whereas $w_2$ is concentrated around the center. The different diffusion coefficients across the interface $x=0$ produce visibly asymmetric propagation: the population fronts advance more rapidly in the left subdomain, where $\mu$ is 10 times larger than in the right subdomain. The sharp transition layers across the discontinuity of the diffusion coefficient are resolved without visible spurious oscillations due to the robustness provided by the SWIP discretization method. 
\par
The lower and upper portions of the domain act as two distinct local regimes for the evolution of $w_2$. In the lower region, the presence of the initially localized $w_1$ concentrations modifies the reaction dynamics of $w_2$ through the coupling factor $(1+w_1)$, whereas the upper region is initially characterized by the standard Fisher--Kolmogorov kinetics. The interaction factor $(1+w_1)$ modifies the local Fisher--Kolmogorov dynamics of $w_2$ in regions occupied by $w_1$. Its expansion contains the cubic contribution $-w_1w_2^2$, which provides a density-dependent saturation mechanism. The snapshots illustrate the qualitative effect of the full coupled reaction dynamics in the presence of discontinuous diffusion.
\section{Conclusions}
\label{sec:conclusion}
In this work, we introduced and analyzed a weighted discontinuous Galerkin method for hierarchically coupled nonlinear reaction--diffusion systems with linear, quadratic, and cubic reaction terms. A dissipative diagonal structure is assumed for the intra-block cubic interactions, providing the key mechanism for controlling the nonlinear reaction contributions. The proposed framework accommodates multiple interacting populations and heterogeneous, possibly discontinuous or anisotropic, diffusion tensors on general polygonal meshes.
\par
The spatial discretization is based on a weighted symmetric interior penalty formulation, for which we establish coercivity and continuity of the discrete diffusion bilinear form. We then derive local-in-time stability estimates for the semi-discrete problem through a recursive argument along the hierarchy of population blocks. In particular, the stability bounds for each block depend explicitly on the estimates obtained for the preceding blocks, consistently with the one-way coupling structure of the model. Moreover, an a priori error estimate in a combined $\mathbf{L}^2(\Omega)$--dG energy norm is proved for sufficiently regular solutions. The numerical experiments support the theoretical results, confirming the expected convergence behaviour for both classical and higher-order reaction--diffusion systems. Moreover, the coupled Fisher--Kolmogorov test with discontinuous diffusion coefficients illustrates the robustness of the weighted interior penalty formulation in the presence of heterogeneous diffusion and nonlinear inter-population coupling.
\par
Extending the analysis to three spatial dimensions remains an open challenge, since the present proof relies on Sobolev embeddings that do not provide the required control of the higher Lebesgue norms arising from cubic reaction terms. Several further directions warrant investigation. A fully discrete stability and error analysis for the IMEX--Runge--Kutta time discretization would complement the present semi-discrete theory. Other natural developments include the design of adaptive $hp$-PolyDG strategies for localized reaction fronts and heterogeneous media, as well as the application of the proposed framework to more detailed multi-species ecological, epidemiological, and neurodegenerative-disease models.

\appendix
\section{Recursive coefficients in the block stability estimates}
\label{app:recursive-coefficients}

This appendix collects the nonnegative coefficient functions entering the
stability estimates for the $p$-th block. Besides shortening the proof of
Theorem~\ref{thm:stab_DG_p}, these definitions make explicit the
recursive dependence on the stability bounds of the preceding blocks. For
$p=1$, all the sums below are understood to be empty and hence equal to zero.

\subsection{Coefficients for the $L^2$-energy estimate}

For $p\ge2$ and almost every $t$, define
\begin{equation}
\label{eq:A-coefficients}
\begin{aligned}
    A_{4,3}^{(p)}(\tau) & =\frac{15^3}{2} \CGN{4}^{12} c_\infty^4 C_{E_4}^4 \sum_{q=1}^{p-1} \Cen{2}{q}(\tau)^2, && \quad A_{3,2}^{(p)}(\tau) = \frac{16\sqrt{10}}{3\sqrt{3}} \CGN{3}^{3} \sum_{q=1}^{p-1} \left( q_\infty C_{E_3}  \sqrt{\Cen{2}{q}(\tau)} \right)^{\frac{3}{2}}, \\
    A_{4,2}^{(p)}(\tau) & = 10 \CGN{4}^{4} c_\infty^2 C_{E_4}^4  \sum_{q,s=1}^{p-1} \Cen{2}{q}(\tau)\Cen{2}{s}(\tau),
    && \quad A_{3,1}^{(p)}(\tau) = \frac{5}{6} \CGN{3}^{\frac{6}{5}} \left(\frac{5}{3}\right)^{\frac{1}{5}} \sum_{q,s=1}^{p-1}\left(4q_\infty C_{E_3}^2 \sqrt{\Cen{2}{q}(\tau)\Cen{2}{s}(\tau)} \right)^{\frac{6}{5}},
    \\
    A_{2,1}^{(p)}(\tau) & = l_\infty \sum_{q=1}^{p-1} \sqrt{\Cen{1}{q}(\tau)},
    && \quad A_{4,1}^{(p)}(\tau) = \frac{3}{4}
    \CGN{4}^{\frac{4}{3}} \left(\frac{5}{2}\right)^{\frac{1}{3}} c_\infty^{\frac{4}{3}}C_{E_4}^{4} \sum_{q,s,r=1}^{p-1} \left[8\Cen{2}{q}(\tau) \Cen{2}{s}(\tau) \Cen{2}{r}(\tau) \right]^{\frac{2}{3}},
\end{aligned}
\end{equation}
These functions collect the terms produced by the Gagliardo--Nirenberg and
Young inequalities in the proof of the first estimate of
Theorem~\ref{thm:stab_DG_p}. We then set
\begin{equation}
\label{eq:K-coefficients}
\begin{aligned}
 K_1^{(p)}(t)
 &:=%
 \frac{8}{27}q_\infty^2 C_{\mathrm{GN},3}^{6}t
 +\int_0^t\left[
   \frac14 A_{2,1}^{(p)}(\tau)
   +\frac{3\sqrt[3]{4}}{5\sqrt[3]{25}}A_{3,1}^{(p)}(\tau)
   +\frac{2}{3\sqrt3}A_{4,1}^{(p)}(\tau)
 \right]d\tau,
\\
 K_2^{(p)}(t)
 &:=%
 A_{2,1}^{(p)}(t)+A_{3,1}^{(p)}(t)+A_{4,1}^{(p)}(t)
 +A_{3,2}^{(p)}(t)+A_{4,2}^{(p)}(t),
\\
K_3^{(p)}(t)
 &:=%
 A_{4,3}^{(p)}(t).
\end{aligned}
\end{equation}
\subsection{Coefficients for the higher-energy estimate}
The terms depending only on the preceding blocks are collected in
\begin{equation}
\label{eq:D-coefficients}
\begin{aligned}
 D_1^{(p)}(t)
 &:=%
 \int_{0}^{t} \left( 6 l_\infty^2\sum_{q=1}^{p-1} \Cen{1}{q}(\tau) + 64 q_\infty^2
    \sum_{q,s=1}^{p-1} \Cen{2}{s}(\tau) \Cen{2}{q}(\tau) C_{E_4}^2 + \sum_{q,s,r=1}^{p-1} 128 c_\infty^2 \Cen{2}{r}(\tau) \Cen{2}{s}(\tau) \Cen{2}{q}(\tau) C_{E_6}^3\right) \dtau,
\\
 D_2^{(p)}(t)
 &:= 24 q_\infty^2 \sum_{q=1}^{p-1} \Cen{2}{q}(t) C_{E_4} + 64 c_\infty^2 \sum_{s,q=1}^{p-1} \Cen{2}{s}(t) \Cen{2}{q}(t) C_{E_8}^2 C_{E_4},
\\
D_3^{(p)}(t)
 &:= 24 c_\infty^2 \sum_{q=1}^{p-1} \Cen{2}{q}(t) C_{E_8}.
\end{aligned}
\end{equation}
The functions in \eqref{eq:K-coefficients} and
\eqref{eq:D-coefficients} are nonnegative and depend only on the already
controlled blocks $1,\ldots,p-1$. Hence, they are known at the $p$th step of
the recursive argument.

\bibliographystyle{hieeetr}
\bibliography{bibliography.bib}

@article{bonetti_robust_2025,
  author  = {Bonetti, Stefano and Botti, Michele and Antonietti, Paola F.},
  title   = {Robust discontinuous {G}alerkin-based scheme for the fully-coupled nonlinear thermo-hydro-mechanical problem},
  journal = {IMA Journal of Numerical Analysis},
  volume  = {45},
  number  = {3},
  year    = {2025}, 
  pages   = {1786--1820},
  doi     = {10.1093/imanum/drae045},
}

@article{antonietti_structure-preserving_2026,
    author  = {Antonietti, P. F. and Corti, M. and G\'omez, S. and Perugia, I.},
    title   = {A structure-preserving {LDG} discretization of the {F}isher-{K}olmogorov equation for modeling neurodegenerative diseases},
    journal = {Mathematics and Computers in Simulation},
    volume  = {241},
    year    = {2026},
    pages   = {351--366},
    doi     = {10.1016/j.matcom.2025.09.006},
}

@article{pierre_schmitt_blowup_2000,
  author  = {Pierre, Michel and Schmitt, Didier},
  title   = {Blowup in Reaction-Diffusion Systems with Dissipation of Mass},
  journal = {SIAM Review},
  volume  = {42},
  number  = {1},
  pages   = {93--106},
  year    = {2000},
  doi     = {10.1137/S0036144599350619},
}

@article{souplet_global_2018,
  author  = {Souplet, Philippe},
  title   = {Global Existence for Reaction--Diffusion Systems with Dissipation of Mass and Quadratic Growth},
  journal = {Journal of Evolution Equations},
  volume  = {18},
  number  = {3},
  pages   = {1713--1720},
  year    = {2018},
  doi     = {10.1007/s00028-018-0440-8},
}

@article{pierre_global_2010,
  author  = {Pierre, Michel},
  title   = {Global Existence in Reaction-Diffusion Systems with Control of Mass: A Survey},
  journal = {Milan Journal of Mathematics},
  volume  = {78},
  number  = {2},
  pages   = {417--455},
  year    = {2010},
  doi     = {10.1007/s00032-010-0133-4},
}

@article{sing_higher_2021,
    title       = {Higher order interactions and species coexistence},
    author      = {Singh, P. and Baruah, G.},
    journal     = {Theoretical Ecology},
    year        = {2021},
    volume      = {14},
    pages       = {71–-83},
    doi         = {10.1007/s12080-020-00481-8},
}

@article{letten_mechanistic_2019,
  author  = {Letten, Andrew D. and Stouffer, Daniel B.},
  title   = {The mechanistic basis for higher-order interactions and non-additivity in competitive communities},
  journal = {Ecology Letters},
  volume  = {22},
  number  = {3},
  pages   = {423--436},
  year    = {2019},
  doi     = {10.1111/ele.13211},
}

@article{grilli_higher-order_2017,
    title       = {Higher-order interactions stabilize dynamics in competitive network models},
    author      = {Grilli, J. and Barabás, G. and Michalska-Smith, M.J. and Allesina, S.},
    journal     = {Nature},
    year        = {2017},
    volume      = {548},
    pages       = {210–-213},
    doi         = {10.1038/nature23273},
}

@article{terry_impact_2025,
    title       = {The impact of structured higher‑order interactions on ecological network stability},
    author      = {Terry, J.C.D. and Bonsall, M.B. and Morris, R.J.},
    journal     = {Theoretical Ecology},
    year        = {2025},
    volume      = {18},
    number      = {9},
    pages       = {1--12},
    doi         = {10.1007/s12080-025-00603-0},
}

@article{wootton_nature_1994,
    title       = {The nature and consequences of indirect effects in ecological communities},
    author      = {Wootton, J.T.},
    journal     = {Annual Review of Ecology, Evolution, and Systematics},
    year        = {1994},
    volume      = {25},
    pages       = {443--466},
    doi         = {10.1146/annurev.es.25.110194.002303},
}

@article{vazquez-palomo_computational_2026,
    title       = {{A computational framework to predict the spreading of Alzheimer’s disease}},
    author      = {Vazquez-Palomo, A. and Betegón, C. and Weickenmeier, J. and Martínez-Pañeda, E.},
    journal     = {Engineering with Computers},
    year        = {2026},
    volume      = {42},
    number      = {78},
    pages       = {1--25},
    doi         = {10.1007/s00366-026-02313-5},
}

@article{antonietti_lymph_2025,
    title       = {{lymph: discontinuous poLYtopal methods for multi-physics differential problems}},
    author      = {Antonietti, P.F. and Bonetti, S. and Botti, M. and Corti, M. and Fumagalli, I. and Mazzieri, I.},
    journal     = {Transactions on Mathematical Software},
    year        = {2025},
    volume      = {51},
    number      = {1},
    pages       = {1--22},
    doi         = {10.1145/3716310},
}

@article{talischi_polymesher_2012,
    author      = {Talischi, C. and Paulino, G. H. and Anderson, P. and Menezes, Ivan F. M.},
    title       = {{\tt {P}oly{M}esher}: a general-purpose mesh generator for polygonal elements written in {M}atlab},
    journal     = {Structural and Multidisciplinary Optimization},
    volume      = {45},
    year        = {2012},
    number      = {3},
    pages       = {309--328},
    doi         = {10.1007/s00158-011-0706-z},
}

@article{chen_exact_2012,
    title       = {Exact travelling wave solutions of three-species competition--diffusion systems},
    author      = {Chen, C-C and Hung, L-H and Mimura, M. and Ueyama, D.},
    journal     = {Discrete and Continuous Dynamical Systems - Series B},
    year        = {2012},
    volume      = {17},
    number      = {8},
    pages       = {2653--2669},
    doi         = {10.3934/dcdsb.2012.17.2653},
}

@misc{antonietti_optimized_2026,
      title         = {Optimized high-order {IMEX-RK} schemes for degenerate diffusion-reaction problems with application to travelling waves phenomena}, 
      author        = {Antonietti, P.F. and Corti, M. and Orlando, G.},
      year          = {2026},
      eprint        = {2606.15726},
      archivePrefix = {arXiv},
      primaryClass  = {math.NA},
      doi           = {10.48550/arXiv.2606.15726},
}

@book{smoller_shock_1994,
    title       = {Shock {Waves} and {Reaction}—{Diffusion} {Equations}},
    publisher   = {Springer},
    author      = {Smoller, J.},
    year        = {1994},
    address     = {New York, NY},
    doi         = {10.1007/978-1-4612-0873-0},
}

@article{babuska_p_1994,
    title       = {The \textit{p} and \textit{h-p} versions of the Finite Element Method, Basic Principles and Properties},
    author      = {Babu\v{s}ka, I. and Suri, M.},
    journal     = {SIAM Review},
    year        = {1994},
    volume      = {36},
    number      = {4},
    pages       = {578--632},
    doi         = {10.1137/1036141},
}

@article{gazca-orozco_stability_2025,
  	title      	= {On the stability and convergence of discontinuous {G}alerkin schemes for incompressible flows},
  	author     	= {Gazca-Orozco, P.A. and Kaltenbach, A.},
  	journal    	= {IMA Journal of Numerical Analysis},
  	year       	= {2025},
  	volume     	= {45},
    number      = {1},
  	pages      	= {243--282},
  	doi        	= {10.1093/imanum/drae004},
}

@article{lotfi_partial_2014,
	title      = {Partial Differential Equations of an Epidemic Model with Spatial Diffusion},
	author     = {Lotfi, E.M. and Maziane, M. and Hattaf, K. and Yousfi, N.},
	journal    = {International Journal of Partial Differential Equations},
	year       = {2014},
	volume     = {1},
	pages      = {186437},
	doi        = {10.1155/2014/186437},
}

@article{he_effects_2013,
	title      = {The effects of diffusion and spatial variation in {Lotka}–{Volterra} competition–diffusion system {I}: {Heterogeneity} vs. homogeneity},
	journal    = {Journal of Differential Equations},
	author     = {He, X. and Ni, W.-M.},
	year       = {2013},
	volume     = {254},
	number     = {2},
	pages      = {528--546},
	doi        = {10.1016/j.jde.2012.08.032},
}

@article{fornari_spatially-extended_2020,
    title       = {Spatially-extended nucleation-aggregation-fragmentation models for the dynamics of prion-like neurodegenerative protein-spreading in the brain and its connectome},
    author      = {Fornari, S. and Sch\"afer, A. and Kuhl, E. and Goriely, A.},
    journal     = {Journal of Theoretical Biology},
    year        = {2020},
    volume      = {486},
    pages       = {110102},
    doi         = {https://doi.org/10.1016/j.jtbi.2019.110102},
}

@article{corti_discontinuous_2023,
	title      = {Discontinuous {Galerkin} methods for {Fisher}–{Kolmogorov} equation with application to α-synuclein spreading in {Parkinson}’s disease},
	author     = {Corti, M. and Bonizzoni, F. and Dede’, L. and Quarteroni, A. M. and Antonietti, P. F.},
	journal    = {Computer Methods in Applied Mechanics and Engineering},
	year       = {2023},
	volume     = {417},
	pages      = {116450},
	doi        = {10.1016/j.cma.2023.116450},
}

@article{fornari_prion-like_2019,
    	title          = {Prion-like spreading of {Alzheimer}'s disease within the brain's connectome},
    	author         = {Fornari, S. and Sch{\"a}fer, A. and Jucker, M. and Goriely, A. and Kuhl, E.},
    	journal        = {Journal of The Royal Society Interface},
    	year           = {2019},
    	volume         = {16},
    	number     	   = {159},
   	    pages     	   = {20190356},
    	doi            = {10.1098/rsif.2019.0356},
}

@book{di_pietro_hybrid_2020, 
	title          = {The {H}ybrid {H}igh-{O}rder {M}ethod for {P}olytopal {M}eshes: {D}esign, {A}nalysis, and {A}pplications},
	author         = {Di Pietro, D. A. and Droniou, J.},
    year           = {2020},
    volume         = {19},
    publisher      = {Springer},
    address        = {Cham},
    address        = {Berlin, Heidelberg},
	doi		       = {10.1007/978-3-030-37203-3}
}

@book{di_pietro_mathematical_2012, 
	title          = {Mathematical {A}spects of {D}iscontinuous {G}alerkin {M}ethods},
	author         = {Di Pietro, D. A. and Ern, A.},
	year           = {2012},    
    volume         = {69},
    publisher      = {Springer},
    address        = {Berlin, Heidelberg},
	doi		       = {10.1007/978-3-642-22980-0}
}

@article{arnoldUnifiedAnalysisDiscontinuous2001,
	author 	= {Arnold, D. N. and Brezzi, F. and Cockburn, B. and Marini, L. D.},
	title 		= {Unified analysis of discontinuous {G}alerkin methods for elliptic problems},
	journal 	= {SIAM Journal on Numerical Analysis},
	year 		= {2002},
	volume 	= {39},
	number 	= {5},
	pages 	= {1749-1779},
	doi 		= {10.1137/S0036142901384162},
}

@book{cangiani_hp_version_2017,
	title 		= {\textit{hp}-{V}ersion {D}iscontinuous {G}alerkin {M}ethods on {P}olygonal and {P}olyhedral {M}eshes},
	author 	= {Cangiani, A. and Dong, Z. and Georgoulis, E. and Houston, P.},
	year 		= {2017},
	publisher   = {Springer},
	doi		= {10.1007/978-3-319-67673-9},
}

@article{ern_discontinuous_2009,
	title      	= {A discontinuous {G}alerkin method with weighted averages for advection--diffusion equations with locally small and anisotropic diffusivity},
	author     	= {Ern, A. and Stephansen, A. F. and Zunino, P.},
	journal    	= {IMA Journal of Numerical Analysis},
	year       	= {2009},
	volume     	= {29},
	number	= {2},
	pages      	= {235-256},
	doi        	= {10.1093/imanum/drm050},
}

@article{webb_extensions_2018, 
    title       = {Extensions of {G}ronwall's inequality with quadratic growth terms and applications},
    author      = {Webb, J.R.L.},
    journal     = {Journal of Qualitative Theory of Differential Equations},
    year        = {2018},
    volume      = {61},
    pages       = {1--12},
    doi         = {10.14232/ejqtde.2018.1.61},

}

\end{document}